\newcommand{\papertitle}{Conservative Finite Element Methods for Kirchhoff-Type
Equations}

\newcommand{\paperkeywords}{Kirchhoff-type equation; nonlinear wave equation; finite
element method; three-layer scheme; Picard iteration; optimal-order error estimate}

\newcommand{\papermsc}{65M60; 65M12; 65M15; 35L70; 35L90}

\newcommand{\paperabstract}{%
We construct and analyse a fully discrete scheme for a Kirchhoff-type nonlinear
integro-differential equation in two spatial dimensions with time-dependent
coefficients. The scheme combines a symmetric three-layer time discretization of
Crank--Nicolson type with a conforming finite element method of arbitrary polynomial
degree $p$ and a Picard iteration for the nonlinearity. We prove that the Picard
iteration is well defined and, that the iterates converge to the Ritz projection of the
exact solution measured in fractional-power norms of the discrete Laplacian. Running the same
energy argument in the unshifted norms, and under the same hypotheses, we then obtain the optimal order $\bigO ( h^{p+1} + \tau^2 )$ in $L^2$ and
$\bigO ( h^{p} + \tau^2 )$ in $H^1$ for the {full} error. For
the homogeneous problem with constant coefficients, the exact solution of the scheme moreover preserves a discrete energy
identically.
Numerical experiments with manufactured solutions reproduce the predicted second order
in time and the full spatial order $h^{p+1}$ in the $L^2$-norm. %
}

\newcommand{\paperdataavailability}{%
The raw error data underlying every figure, together with the sources of all figures, are
contained in the \texttt{figures/} directory of the accompanying repository, \REPOSITORYURL.
The \texttt{software/} directory of the same repository contains the implementation that
produced the reported runs: a \textsc{Python} package built on \textsc{FEniCSx}/\textsc{DOLFINx}.
}

\newcommand{\paperdeclarationAI}{During the preparation of this work, the authors used Claude in order to improve formulation for parts of the manuscript and for parts of the code to create all-encompassing scripts for the various examples and to improve presentation of the software. After using this tool, the authors reviewed and edited the content as needed and take full responsibility for the content of the published article.}

\newcommand{\papersubject}{Numerical analysis of conservative fully discrete finite element methods for Kirchhoff-type integro-differential equations with time-dependent coefficients}

\documentclass[a4paper]{article} 

\newif\ifsubmission  \submissiontrue
\newif\ifblind       \blindfalse

\usepackage[
a4paper,                          
margin=28mm,                      
marginparwidth=22mm,              
includehead, includefoot          
]{geometry}

\usepackage[dvipsnames]{xcolor}                 

\usepackage{multicol}               

\usepackage{longtable,booktabs}     

\usepackage{graphicx}               
\graphicspath{{figures/}}           

\usepackage[section]{placeins}
\usepackage{standalone}
\usepackage{tikz}
\usepackage{pgfplots}
\pgfplotsset{compat=1.18}

\usepackage{amsmath}                
\usepackage{amssymb}                
\usepackage{amsthm}                 
\usepackage{mathtools}              
\usepackage{mathrsfs}

	\usepackage{silence}                
	\usepackage[T1]{fontenc}                 
	\usepackage[utf8]{inputenc}              
	\usepackage{lmodern}                     
	\usepackage[nopatch=footnote]{microtype} 
	
	\usepackage{ragged2e}
	
	\usepackage[normalem]{ulem}         
	
	\usepackage{titlesec}               
	\usepackage{setspace}               
	\usepackage{indentfirst}            
	
	\usepackage[inline]{enumitem}
	
	\newcommand{\authorDoerich}{Benjamin D\"orich}
	\newcommand{\authorKrumbiegel}{Felix Krumbiegel}
    \newcommand{\authorRogava}{Jemal Rogava}
	\newcommand{\authorRupp}{Andreas Rupp}
	\newcommand{\authorVashakidze}{Zurab Vashakidze}
	
	\newcommand{\orcidDoerich}{0000-0001-5840-2270}
	\newcommand{\orcidKrumbiegel}{0009-0002-7288-4071}
    \newcommand{\orcidRogava}{0000-0001-9460-4283}
	\newcommand{\orcidRupp}{0000-0001-5527-7187}
	\newcommand{\orcidVashakidze}{0000-0001-8736-6213}
	
	\newcommand{\paperauthors}{\authorDoerich{}; \authorKrumbiegel{}; \authorRogava{}; \authorRupp{}; \authorVashakidze{}}

	\newif\ifrepositoryresolved  \repositoryresolvedtrue
	\ifrepositoryresolved
		\newcommand{\REPOSITORYURL}{\url{https://github.com/FelixKrumbiegel/Kirchhoff-conservative-FEM}}
	\else
		\ifblind
			\newcommand{\REPOSITORYURL}{\texttt{[repository URL withheld for double-blind review]}}
		\else
			\newcommand{\REPOSITORYURL}{\url{https://github.com/FelixKrumbiegel/Kirchhoff-conservative-FEM}}
		\fi
		\AtEndDocument{%
			\GenericWarning{}{%
				*** SUBMISSION BLOCKER: the Data and Code Availability statement still contains a
				placeholder. Set \string\repositoryresolvedtrue\space and fill in
				\string\REPOSITORYURL\space with the persistent archive DOI before submitting ***%
			}%
		}
	\fi
	
	\ifblind
		\newcommand{\runauthors}{\runningtitle}
	\else
		\newcommand{\runauthors}{\authorDoerich{} \textbullet\ \authorKrumbiegel{} \textbullet\ \authorRogava{} \textbullet\ \authorRupp{} \textbullet\ \authorVashakidze{}}
	\fi
	\newcommand{\runningtitle}{\papertitle} 
	
	\usepackage{fancyhdr}               
	\AtBeginDocument{\thispagestyle{empty}} 
	
	\fancypagestyle{plain}{%
		\fancyhf{}
		\fancyhead[C]{%
			\ifodd\value{page}%
			\small \runauthors%
			\else%
			\small \runningtitle%
			\fi%
		}
		\fancyfoot[C]{\thepage}

	}
	
	\definecolor{linkColor}{RGB}{155,0,119} 
	\ifblind
		\newcommand{\pdfauthorfield}{}
	\else
		\newcommand{\pdfauthorfield}{\paperauthors}
	\fi
	\usepackage[
	pdfstartview  = {FitH},          
	pdftitle      = {\papertitle},   
	pdfauthor     = {\pdfauthorfield}, 
	pdfkeywords   = {\paperkeywords},  
	pdfsubject    = {{\papersubject}},   
	colorlinks    = true,            
	linkcolor     = linkColor,       
	citecolor     = linkColor,       
	urlcolor      = linkColor,       
	hypertexnames = false            
	]{hyperref}                         
	
	\usepackage{orcidlink} 
	
	\usepackage[capitalise, noabbrev, nameinlink]{cleveref} 
	
	\crefname{equation}{eq.}{eqs.}      
	\Crefname{equation}{Eq.}{Eqs.}      
	\crefname{theorem}{theorem}{theorems}
	\Crefname{theorem}{Theorem}{Theorems}
	\crefname{figure}{figure}{figures}
	\Crefname{figure}{Figure}{Figures}
	\crefname{table}{table}{tables}
	\Crefname{table}{Table}{Tables}
	
	\theoremstyle{plain}
	
	\newtheorem{theorem}{Theorem}[section]      
    \crefname{theorem}{Theorem}{Theorems}   
	
	\newtheorem*{theorem*}{Theorem}             
	
	\newtheorem{lemma}[theorem]{Lemma}          
	
	\newtheorem*{lemma*}{Lemma}                 
	
	\newtheorem{corollary}[theorem]{Corollary}  
    \crefname{corollary}{Corollary}{Corollaries}   
	
	\newtheorem*{corollary*}{Corollary}         
	
	\theoremstyle{definition}
	
	\newtheorem*{definition*}{Definition}        
	
	\newtheoremstyle{boldremark} 
	{3pt}                      
	{3pt}                      
	{\normalfont}              
	{}                         
	{\bfseries}                
	{.}                        
	{.5em}                     
	{}                         
	
	\theoremstyle{boldremark}
	
	\newtheorem{remark}[theorem]{Remark}        
	
	\newtheorem*{remark*}{Remark}               
	
	\theoremstyle{definition}

	\newtheorem{assumption}[theorem]{Assumption}
  \Crefname{assumption}{Assumption}{Assumptions}
  \crefname{assumption}{Assumption}{Assumptions}

	\newtheorem{hypothesis}[theorem]{Hypothesis}
  \Crefname{hypothesis}{Hypothesis}{Hypothesis}
  \crefname{hypothesis}{Hypothesis}{Hypothesis}

	\theoremstyle{plain}
	\newtheorem{proposition}[theorem]{Proposition}
	\theoremstyle{definition}
	
	\newcommand{\refitem}[2]{%
		\hyperref[#2]{#1~\ref*{#2}}
	}
	
	\newcommand{\figsubref}[2]{%
		\hyperref[#2]{Figure~\ref*{#1}(\subref*{#2})}
	}
	
    \newcommand{\appref}[1]{%
        \hyperref[#1]{Appendix}%
    }
	
	\makeatletter
	\renewenvironment{proof}[1][\proofname]{
		\par\pushQED{\qed}\normalfont
		\topsep6\p@\@plus6\p@\relax
		\trivlist\item[\hskip\labelsep\bfseries#1\@addpunct{.}]
		\ignorespaces
	}{%
		\popQED\endtrivlist\@endpefalse
	}
	\makeatother
	
	\newcommand{\R}{\mathbb{R}}         
    \newcommand{\domain}{\Omega}
    \newcommand{\mesh}{\mathcal{T}_h}
    \newcommand{\ltwoprojection}{P_h}
    \newcommand{\ritzprojection}{R_h}
    \newcommand{\discreteL}{\mathcal{L}_h}
    
    \newcommand{\interpolation}{J_h}
    \newcommand{\cstinv}{C_{\mathrm{inv}}} 
    \newcommand{\cstcoerc}{C_{\theta}}     

	\long\def\bigO{\mathcal{O}}        
	\long\def\L{\mathcal{L}}           
	\long\def\T{\mathcal{T}}           
	\long\def\ie{\textit{i.e.}}        
	\long\def\apriori{\textit{a priori}} 
	\long\def\wrt{\textit{wrt.}} 
	
	\newcommand{\dint}[2][\Omega]{
		\iint\limits_{#1}
		#2
		\,\mathrm{d}x
		\,\mathrm{d}y
	}
	
\newcounter{constant}

\newcounter{mconst}

\newcommand{\cst}[1]{%
  \refstepcounter{constant}
  \expandafter\xdef\csname #1\endcsname{c_{\arabic{constant}}}%
}

\newcommand{\mcst}[1]{%
  \refstepcounter{mconst}
  \expandafter\xdef\csname #1\endcsname{M_{\arabic{mconst}}}%
}

\mcst{cstmone}    
\mcst{cstmtwo}    
\mcst{cstmthree}  
\mcst{cstmfour}   

\mcst{cstmfive}   
\mcst{cstmseven}  
\mcst{cstmseight}  
\mcst{cstmsix}    

\cst{lmonecstalph}   
\cst{lmonecstbeta}   
\cst{lmonecstdiff}   

\cst{lmtwocstdiffbw}             
\cst{lmtwocsthatenerg}           
\cst{lmtwocstfracone}            
\cst{lmtwocstfractwo}            
\cst{lmtwocstmaxf}               
\cst{lmtwocstnearestexp}         
\cst{lmtwocstnearestfinineq}     

\cst{cststatementone}            
\cst{cststatementtwo}            
\cst{thmcstdiffbetaw}            
\cst{thmcsthalflnorm}            
\cst{thmcstlnorm}                
\cst{thmcstinnerprodgk}          
\cst{thmcstinnerprodrem}         
\cst{thmcstsumoverczero}         
\cst{thmcstmaxtwoconst}          
\cst{thmcstesttone}              
\cst{thmcstestwone}              
\cst{thmtildevareps}             
\cst{thmtildeEkplusone}          

\cst{rmkthmcstkinetic}           

\cst{thmcstesttildc}             
\cst{thmcstsqrtekplusone}        

\cst{cstczero}    
\cst{cstcone}     
	
	\titleformat{\section}{\large\bfseries}{\thesection}{1em}{}            
	\titleformat{\subsection}{\normalsize\bfseries}{\thesubsection}{1em}{} 
	
	\titlespacing*{\section}{0pt}{*2}{*0.6}
	\titlespacing*{\subsection}{0pt}{*1.3}{*0.4}
	
	\usepackage[%
                doi=false,
                backend=biber,
                style=numeric,
                ]{biblatex} 
	\appto\biburlsetup{\Urlmuskip=0mu plus 1mu\relax}
	\numberwithin{equation}{section} 
	\numberwithin{figure}{section}   
	\numberwithin{table}{section}    
	
	\usepackage[most]{tcolorbox} 
	\tcbset{%
		myabstract/.style={%
			colback=gray!5!white,    
			colframe=black!40,       
			fonttitle=\bfseries,     
			coltitle=black,          
			boxrule=0.6pt,           
			arc=2mm,                 
			enhanced,                
			before skip=10pt,        
			after skip=20pt,         
			breakable                
		}%
	}
	\renewenvironment{abstract}
	{\begin{tcolorbox}[myabstract,title=Abstract]}
		{\end{tcolorbox}}
	
	\title{\papertitle}                 
	\author{\runauthors}                
	\date{}                             
	
	\newcommand{\titleseparator}{\par\noindent{\color{black}\rule{\linewidth}{0.8pt}}\par}
	
	\allowdisplaybreaks

	\AtBeginDocument{%
		\setlength{\abovedisplayskip}{5pt plus 2pt minus 2pt}%
		\setlength{\belowdisplayskip}{5pt plus 2pt minus 2pt}%
		\setlength{\abovedisplayshortskip}{2pt plus 1pt minus 1pt}%
		\setlength{\belowdisplayshortskip}{3pt plus 1pt minus 1pt}}
	\makeatletter\def\thm@space@setup{\thm@preskip=4pt \thm@postskip=4pt}\makeatother  

\begin{document}

		
		\sloppy  
		
        \begingroup\thispagestyle{empty}
        	\thispagestyle{empty} 
        	
        	\begin{center}
        		\medskip
        		
        		\titleseparator 
        		
        		{\LARGE\bfseries \papertitle\par} 
        		
        		\titleseparator 
        		
        		\medskip
        		
        		%

        		\ifblind
\else
        		\begin{tabular}[t]{cc}
        			
        			\begin{minipage}[t]{0.44\textwidth}
        				\centering 
        				
        				{\large\textbf{\authorDoerich{}}~\orcidlink{\orcidDoerich}\par} 
        				
        				{\small
      						Department of Mathematics, Saarland University,\\
        				Campus E1.1, 66123 Saarbrücken, Germany\par
        				} 
        				
        				
        				{\small
        					\href{mailto:benjamin.doerich@uni-saarland.de}{\texttt{benjamin.doerich@uni-saarland.de}}\par
        				} 
        			\end{minipage}
        			&
        			
        			\begin{minipage}[t]{0.44\textwidth}
        				\centering 
        				
        				{\large\textbf{\authorKrumbiegel{}}~\orcidlink{\orcidKrumbiegel}\par} 
        				
        				{\small
        					Department of Mathematics, Saarland University,\\
        					Campus E1.1, 66123 Saarbrücken, Germany\par
        				} 
        				
        				
        				{\small
        					\href{mailto:felix.krumbiegel@uni-saarland.de}{\texttt{felix.krumbiegel@uni-saarland.de}}\par
        				} 
        			\end{minipage}
        			
        		\end{tabular}
        		
        		
        		\begin{tabular}[t]{cc}
        			
        			\begin{minipage}[t]{0.44\textwidth}
        				\centering 
        				
        				{\large\textbf{\authorRogava{}}~\orcidlink{\orcidRogava}\par} 
        				
        				{\small
        					Ivane Javakhishvili Tbilisi State University, Ilia Vekua Institute of Applied Mathematics,\\
        					11 University Street, 0186, Tbilisi, Georgia\par
        				} 
        				
        				
        				{\small
        					\href{mailto:jemal.rogava@tsu.ge}{\texttt{jemal.rogava@tsu.ge}}\par
        				} 
        			\end{minipage}
        			&
        			
        			\begin{minipage}[t]{0.44\textwidth}
        				\centering 
        				
        				{\large\textbf{\authorRupp{}}~\orcidlink{\orcidRupp}\par} 
        				
        				{\small
        					Department of Mathematics, Saarland University,\\
        					Campus E1.1, 66123 Saarbrücken, Germany\par
        				} 
        				
        				
        				{\small
        					\href{https://rupp.ink}{\texttt{rupp.ink}} \textbar\ %
        					\href{mailto:andreas.rupp@uni-saarland.de}{\texttt{andreas.rupp@uni-saarland.de}}\par
        				} 
        			\end{minipage}
        		\end{tabular}
        		
        		\begin{tabular}[t]{cc}
        			
        			\begin{minipage}[t]{0.44\textwidth}
        				\centering 
        				
        				{\large\textbf{\authorVashakidze{}}~\orcidlink{\orcidVashakidze}\par} 
        				
        				{\small
        					School of Science and Technology, The University of Georgia (UG),\\
        					77 Merab Kostava Street, 0171, Tbilisi, Georgia;\\[0.4em]
        					Ilia Vekua Institute of Applied Mathematics,\\
        					11 University Street, 0186, Tbilisi, Georgia\par
        				} 
        				
        				
        				{\small
        					\href{mailto:z.vashakidze@ug.edu.ge}{\texttt{z.vashakidze@ug.edu.ge}} \textbar\ %
        					\href{mailto:zurab.vashakidze@tsu.ge}{\texttt{zurab.vashakidze@tsu.ge}}\par
        				} 
        			\end{minipage}
        			&
        			
        			\begin{minipage}[t]{0.44\textwidth}
        				\centering 
        				
        				{\large\phantom{\textbf{\authorVashakidze{}}~\orcidlink{\orcidVashakidze}}\par} 
        			\end{minipage}

        		\end{tabular}
        		\fi 

        		\medskip
        		
        		\begin{minipage}{0.9\textwidth}
        			
        			\begin{abstract}
				\paperabstract
        				
        			\end{abstract}

        			\medskip

        			\noindent\textbf{Keywords:} \paperkeywords.

        			\noindent\textbf{MSC 2020:} \papermsc.

        		\end{minipage}
        		
        	\end{center}
        \endgroup
		
		\newpage

\section{Introduction}\label{sec:introduction}

\subsection{Problem Formulation}\label{sec:problem_formulation}

Kirchhoff-type equations constitute an important class of nonlocal partial differential
equations arising from the classical model proposed by Kirchhoff as an extension of the linear
theory of vibrating strings; see \cite{kirchhoff1876vorlesungen}. In this model the tension
varies according to the total stretching induced by the displacement, so that the principal
part of the equation depends on a global quantity involving the solution itself, typically its
Dirichlet energy. In two spatial dimensions, the equation models the transverse vibrations of
an elastic membrane whose effective tension depends on its total deformation energy, a setting
that arises in contemporary work on nanomechanical oscillators~\cite{Yao2022}; allowing the
material parameters to depend on time reflects the variation of those properties during
vibration, for instance under thermal fluctuation or strong external
driving~\cite{Samanta2023,Sun2006}.

We consider the nonlinear Kirchhoff-type integro-differential equation
\begin{equation}\label{eq:main_kirchhoff}
  \frac{\partial^2 u  }{\partial t^2}
  - \Bigl( \alpha(t)  + \beta(t)  \dint{\left| \nabla u   \right|^2} \Bigr) \Delta u  
  = f  \,,
\end{equation}
posed in $ \left[ 0,T \right]$ and a domain $\domain \subset \R^2$, together with the
initial and boundary conditions
\begin{equation}\label{eq:initial_conds}
  u ( 0 ) = u_0\,,\qquad
  \partial_t u ( 0 ) = v_0\,,\qquad
  u ( t ) |_{\partial \domain} = 0 \quad \text{for } t \in \left[ 0,T \right]\,.
\end{equation}
Here $u$ is the transverse displacement, $\left| \, \cdot\, \right|$ the Euclidean norm and
$\nabla$ the gradient. Throughout both this paper and the companion
paper~\cite{KirchhoffTime}, $u$ and $f$ are regarded as functions of $t$ with values in a space
of functions on $\domain$, so that the spatial argument is never displayed. $u ( t )$ is the
state at time $t$ and $u_h^{n}$ the fully discrete approximation at the $n$-th time level. The coefficients are scalar functions of $t$ and do carry that
argument, \ie, $\alpha, \beta \in C^1 \left[ 0,T \right]$ satisfy $\alpha_{\min} > 0$ and
$\beta_{\min} > 0$, where throughout $\alpha_{\min} = \min_{[0,T]} \alpha ( t )$,
$\alpha_{\max} = \max_{[0,T]} \alpha ( t )$ and $\beta_{\min}, \beta_{\max}$ are defined
correspondingly. The source term is continuous in time with values in $L^2 ( \domain )$, so that
the grid values $f ( n\tau )$, where $\tau$ is the time step size, are defined. \Cref{sec:full-discretization} imposes more regularity on the right-hand side where it is needed. Throughout this
paper $\domain$ is a bounded convex \emph{polytopal} domain, so that it is meshed exactly by
the simplicial families used below. For a general bounded convex domain, the analysis applies
verbatim to a polytopal approximation $\domain_h \subset \domain$, at the cost of an additional
geometric consistency error, or with isoparametric elements of matching degree, and we do not
pursue either variant here.

\subsection{Related work and contribution}\label{sec:related_work}

Numerical methods for Kirchhoff-type equations have a long history. Christie and Sanz-Serna~\cite{Christie1984} combined a finite element discretization in space with a Crank--Nicolson scheme in time for the one-dimensional problem with constant coefficients. In \cite{Peradze2005}, Peradze proposed and analyzed a numerical scheme combining a sine-based Galerkin spatial discretization with a Crank--Nicolson-type time-stepping method for the nonlinear one-dimensional {homogeneous} Kirchhoff string problem with constant coefficients. Gudi \cite{Gudi2012} investigated a finite-element approximation of a second-order nonlinear elliptic boundary value problem that arises as the stationary counterpart of the generalized Kirchhoff equation. More recently, Mallik \cite{Mallik2026} developed and analyzed a hybrid high-order finite element method for a nonlocal nonlinear elliptic boundary value problem of Kirchhoff type, establishing optimal-order convergence in the discrete energy norm. Rogava and Tsiklauri~\cite{RogTsikl2012,RogTsikl2014} analysed symmetric three-layer semi-discrete schemes in an abstract setting and established local convergence; the one-dimensional problem with \emph{time-dependent} coefficients was treated in \cite{RogVashZAMM2024}, and related schemes appear in \cite{RogTsikVash2023,VashGMJ2022,Vash2020}. Well-posedness of the continuous problem for constant coefficients goes back to Arosio and Panizzi~\cite{wellposedness_kirchhoff_1996}.

Beyond discretizations developed specifically for Kirchhoff equations, there has also been extensive work on the numerical analysis of quasilinear wave problems. Classical finite-element and multistep analyses for nonlinear second-order hyperbolic equations were given by Bales and Dougalis \cite{BalD89} and Makridakis \cite{Mak93}. More recently, spatial and fully discrete error analyses for general classes of quasilinear wave-type equations were developed by Hochbruck and Maier \cite{HocM22} and Maier \cite{Mai22}, while strong-norm estimates under weak CFL-type conditions were obtained by D\"orich \cite{Doe24}. Related finite-element analyses for nonlinear acoustic wave models include Nikoli\'{c} and Wohlmuth \cite{NicW19} and D\"orich and Nikoli\'{c} \cite{DoeN26}. These works concern different, predominantly local or damped quasilinear structures and different time integrators. In contrast, the coefficient considered here is a time-dependent spatially nonlocal scalar depending on the total Dirichlet energy, and the symmetric three-layer discretization is designed to preserve an exact discrete energy in the homogeneous constant-coefficient case.

The companion paper~\cite{KirchhoffTime} analyses the temporal discretization of
\eqref{eq:main_kirchhoff}--\eqref{eq:initial_conds} in two space dimensions: it establishes
energy conservation for two symmetric three-layer schemes, uniform \apriori{} bounds, local
second-order convergence in time. The present paper takes the semi-discrete scheme
analysed there as its starting point and supplies the spatial discretization. The
contributions are the following. %

We formulate a fully discrete scheme combining the symmetric three-layer time
stepping with a conforming finite element method of arbitrary polynomial degree $p$ and a
Picard iteration for the nonlinearity. The discretization of the nonlinear coefficient leaves
one free parameter, and our first result is that for the symmetric weighting in time and
time-independent coefficients, each choice of that parameter has its own discrete energy,
which the scheme conserves exactly on the homogeneous problem
(\Cref{prop:fully_discrete_conservation}). %
Further, we prove that at each time level the nonlinear scheme admits a unique solution, to
which the Picard iteration converges (\Cref{lem:fully_discrete_local_existence}). We then
prove convergence of the error relative to the Ritz projection of the exact solution, measured
in fractional-power norms of the discrete Laplacian, locally in time
(\Cref{thm:fully_discrete_local_convergence}) and then over all time levels
(\Cref{cor:fully_discrete_convergence}). A single argument covers a whole range of these
norms. The stronger pair is what the existence theory needs and yields
$\bigO ( h^{p+\frac12} + \tau^{2} )$, while the weaker pair yields the optimal order. To our
knowledge this step has not been taken for Kirchhoff-type equations with time-dependent
coefficients. %
From the latter we obtain an estimate of \emph{optimal} order for the full error
$u(t_n) - u_h^n$, namely $\bigO ( h^{p+1} + \tau^{2} )$ in $L^2$ and
$\bigO ( h^{p} + \tau^{2} )$ in $H^1$ (\Cref{cor:fully_discrete_optimal_error}). The
half-order gap is a property of the norms in which the first estimate is measured, not of the
scheme. This is the rate the numerical experiments exhibit.

Two works are
the natural points of comparison. The work \cite{Christie1984} treats the one-dimensional problem with
constant coefficients by a Galerkin method in space and a Crank--Nicolson discretization in
time. Closer to the present setting, Dautov and Ivanova~\cite{DautovIvanova2023} construct a
conservative finite element scheme for the Kirchhoff equation: an implicit two-layer
scheme for the first-order formulation, which
conserves a discrete energy, is realised by Newton's method, and is reported to achieve
$\bigO ( h^{2} + \tau^{2} )$ accuracy in $L^2$ for both unknowns in numerical tests.
The present work differs from these in the following respects. We treat the case of time-dependent coefficients and employ a symmetric three-layer time-discretization scheme for the second-order formulation, coupled with a finite-element spatial discretization of arbitrary polynomial degree $p$. For the resulting nonlinear system, we establish the existence and uniqueness of the solution and prove the convergence of the associated Picard iteration.

The paper is organized as follows: \Cref{sec:setting} fixes notation, states the standing
hypothesis on the continuous problem and records what \cite{KirchhoffTime} proves about the
semi-discrete scheme. \Cref{sec:full-discretization} contains the fully
discrete scheme with its energy conservation, and \Cref{sec:analysis} its convergence analysis. \Cref{sec:numerics} reports the numerical
experiments and \Cref{sec:conclusion} concludes.

\section{Setting and the Semi-Discrete Scheme}\label{sec:setting}

\noindent We write $\L = -\Delta$ for the Dirichlet Laplacian on $\domain$, a self-adjoint
positive definite operator with domain
$D \left( \L \right) = H^2 \left( \domain \right) \cap H_0^1 \left( \domain \right)$, the last
equality holds because $\domain$ is bounded and convex. Its fractional powers $\L^{a}$ are defined
through the spectral resolution. Denote with $\| \cdot \|$ and $\langle \cdot,\cdot \rangle$ the
norm and inner product of $L^2 \left( \domain \right)$, where~$\langle \cdot,\cdot\rangle$ also carries the
$H^{-1} ( \domain )$ -- $H_0^1 ( \domain )$ duality where the first argument calls for it, as noted
at its first occurrence. Interpolation between
$L^2 \left( \domain \right)$ and $D \left( \L \right)$ gives the embedding
\begin{equation}\label{eq:domain_sobolev_embedding}
  D \bigl( \L^{a} \bigr) \hookrightarrow H^{2a} \left( \domain \right)\,,\qquad 0 \leq a \leq 1\,,
\end{equation}
which is the only direction used below: it lets results quoted from the literature in the
Sobolev scale be applied under the operator-domain hypotheses stated here. The reverse
inclusion is false for $a > \tfrac14$, since membership of $D ( \L^{a} )$ encodes the
homogeneous Dirichlet condition. Temporal regularity hypotheses are therefore stated in the
operator domains, spatial approximation hypotheses in the Sobolev scale.

First, we comment on the genuinely nonlinear of the two three-layer schemes considered in
\cite{KirchhoffTime} and the main result therein. It is the one whose fully discrete counterpart we analyse. %
On the uniform temporal grid $t_n = n \tau$, $\tau = \frac{T}{N}$ with the number of times steps~$N\in\mathbb{N}$, the semi-discrete scheme 
\begin{equation}\label{eq:nonlinear_scheme_extended}
  \frac{u^{n+1} - 2 u^{n} + u^{n-1}}{\tau^2}
  + \tfrac{1}{2} \Bigl( \alpha^n + \tfrac{\beta^n}{2}
    \left( \| \nabla u^{n+1} \|^2 + \| \nabla u^{n-1} \|^2 \right) \Bigr)
    \L \left( u^{n+1} + u^{n-1} \right)
  = f^{n}\,
\end{equation}
is given
for $n = 1, \dots, N-1$, with $\alpha^n = \alpha ( t_n )$, $\beta^n = \beta ( t_n )$ and
$f^{n} = f ( t_n )$, and with the starting values
\begin{equation}\label{eq:starting_vec1}
  u^{0} = u_0\,,\qquad
  u^{1} = u_0 + \tau v_0 + \frac{\tau^2}{2} w_0\,,\qquad
  w_0 = f^{0} - q^0 \L u_0\,,\quad
  q^0 = \alpha^0 + \beta^0 \| \nabla u_0 \|^2 \,.
\end{equation}

\begin{proposition}[{\cite[Thm.~5.4, Cor.~5.5 and Cor.~5.6]{KirchhoffTime}}]\label[proposition]{prop:conver_thm_nonlinear}
  Under the following regularity hypotheses $u_0 \in D ( \L^{2} )$, $v_0 \in D ( \L )$, $f \in C ( [0,T] ; D ( \L^{\frac12} ) )$ with
  $f^0 \in D ( \L )$, three time derivatives of $u$ with $\partial_t^3 u$ Lipschitz in $t$, and
  the corresponding hypotheses on the second spatial derivatives, the time-discrete solution
  of \eqref{eq:nonlinear_scheme_extended} converges to the exact solution with order two in
  time, locally on an interval $[0,\overline{T}]$ whose length is determined by the data. The
  same order holds for the central-difference approximation of the first time derivative.
\end{proposition}

\begin{remark}[What is {not} imported from the companion paper]\label[remark]{rem:no_import}
  \Cref{prop:conver_thm_nonlinear} is recorded to place the present paper next to its
  companion, and is never invoked below. Neither is the existence and uniqueness of the time-discrete trajectory. Nothing below requires either. The hypotheses are stated
  with each result and not inherited, in particular the existence of the fully discrete
  solution is proved here and not assumed. 
\end{remark}

We close this section with an existence result for the PDE \eqref{eq:main_kirchhoff}--\eqref{eq:initial_conds}. The analysis below mainly uses the following assumptions, and only for convergence, it is necessary that higher regularity in time and space is assumed. 

\begin{assumption}[Existence of a solution]\label[assumption]{ass:existence_exact_solution}
  Let $u_0 \in D ( \L^{\frac34} )$, $v_0 \in D ( \L^{\frac14} )$ and
  $f \in C^0 ( [0,T_\mathrm{max}] ; D ( \L^{\frac14} ) )$ for some possibly infinite final time~$T_\mathrm{max}$. Then there exists a unique solution
  $u$ of \eqref{eq:main_kirchhoff}--\eqref{eq:initial_conds} with
  \begin{equation*}
    u \in C^0 \bigl( [0,T] ; D ( \L^{\frac34} ) \bigr)
      \cap C^1 \bigl( [0,T] ; D ( \L^{\frac14} ) \bigr)
  \end{equation*}
  for a final time $T < T^{\star\star} \leq T_\mathrm{max}$. 
\end{assumption}

\noindent For constant $\alpha$ and $\beta$ a result of this type is proved in
\cite[Thm.~2.1]{wellposedness_kirchhoff_1996}, where the blow-up time $T^{\star\star}$ is also specified. %
For time-dependent $\alpha$ and $\beta$ the coefficient
$\alpha(t) + \beta(t) \| \nabla u(t) \|^2$ is nonautonomous and that theorem does not apply, so
there we take the statement as a hypothesis on the continuous problem. It is a hypothesis
about the exact solution only; it is not used to construct anything, so no circularity arises.
As in \cite{KirchhoffTime}, we do not claim that the hypotheses of that theorem --- in
particular whatever regularity it imposes on the domain --- are met by every domain admitted
here; where they are not, the statement above is to be read as a hypothesis in the
constant-coefficient case as well.


\section{Extension to full discretization and energy conservation}\label{sec:full-discretization}
\subsection{Preliminaries}
\noindent Following the assumption and the existence of a solution~\Cref{ass:existence_exact_solution} we may define
\begin{alignat*}{2}
M_a  &= 2 \| \L^{a} u \|_{L^\infty(L^2)}, \qquad &&0 \leq a \leq \tfrac34,\\
M'_a &= 2 \| \L^{a} \partial_t u \|_{L^\infty(L^2)}, 
\qquad
C_a = \| \L^{a} f \|_{L^\infty(L^2)}, \qquad &&0 \leq a \leq \tfrac14
\end{alignat*}
for the solution~$u$ and source term~$f$. %
\par The balls around the solution play a key role in the analysis, as we prove that the fully discrete solution lies inside these balls at all times, which is then used to show unique existence and convergence of the solution. 

Let~$\domain\subset \mathbb{R}^2$ be an open, bounded, and convex domain (hence Lipschitz), as assumed in the introduction, and \emph{polytopal}, that is, a convex polygon, and~$\mesh$ be a regular decomposition of~$\domain$ into quasi-uniform triangles. %
Polytopality makes the decomposition into straight-sided simplices exact, so that
$V_h \subset H_0^1(\domain)$ holds without a variational crime. Convexity secures
the $H^2$-regularity behind the embedding~\eqref{eq:domain_sobolev_embedding}. %
Further, let for any polynomial degree~$p \geq 1$ the space of piecewise (\wrt{} the mesh) polynomials up to degree~$p$ with zero boundary conditions be denoted by 
\begin{equation*}
  V_h
  \coloneqq
  \{
    v\in C(\bar\domain)
    \mid
    v \vert_{E} \in \mathbb{P}_p (E) \textrm{ for any element } E\in\mesh
  \}
  \cap
  H^1_0(\domain).
\end{equation*}
Now, we define the discrete operator~$\discreteL\colon V_h \to V_h$ by 
\begin{equation}\label{eq:discrete_laplacian}
  \langle
    \discreteL v_h, \varphi_h
  \rangle
  =
  \langle
    \nabla v_h, \nabla \varphi_h
  \rangle
  \qquad \textrm{ for all } v_h, \varphi_h \in V_h.
\end{equation}
The operator~$\discreteL$ plays the role of~$\L$ for the discrete setting. This connection can be observed in the following. %
Let~$\ltwoprojection$ denote the~$L^2$-orthogonal projection onto~$V_h$, let~$\ritzprojection$
denote the Ritz projection onto~$V_h$, given by
\begin{align*}
  \langle \ltwoprojection v, \varphi_h \rangle = \langle v , \varphi_h \rangle,
  \quad
  \langle \nabla (\ritzprojection v), \nabla \varphi_h \rangle
  = \langle \nabla v , \nabla \varphi_h \rangle,
  \qquad
  \textrm{for all } \varphi_h\in V_h,
\end{align*}
and let~$\interpolation$ denote a quasi-interpolation operator onto~$V_h$: for the lowest order
case~$p=1$ a Cl\'ement-type operator~\cite{Clement1975}, and for the higher-order cases an
Ern--Guermond-type operator~\cite{ErnGuermond2017} with the homogeneous boundary condition built
in. Then, we have that the discrete operator~$\discreteL$ coincides with~$\L$ weakly on~$V_h$,
\ie, for any~$v\in H^1_0(\domain)$ we have that (with $\langle \L v, \varphi_h \rangle$
understood as the duality pairing between $H^{-1}(\domain)$ and $H_0^1(\domain)$, which reduces
to the $L^2$-inner product when $v \in D ( \L )$)
\begin{equation}\label{eq:ritz_commutation}
  \langle \L v,\varphi_h \rangle
  = \langle \nabla  v,\nabla \varphi_h \rangle
  = \langle \nabla \ritzprojection v,\nabla \varphi_h \rangle
  = \langle \discreteL \ritzprojection v,\varphi_h \rangle,
  \qquad
  \textrm{ for all }\varphi_h\in V_h ,
\end{equation}
and in particular, for $v \in D ( \L )$,
\begin{equation}\label{eq:fully_discrete_discrete_laplacian_equality}
  \discreteL \ritzprojection v = \ltwoprojection \L v .
\end{equation}

We collect next the approximation and stability properties of these three operators
that the analysis uses. Their proofs and the references for them are given
in~\Cref{app:projections}. Throughout, the constants implied by~$\lesssim$ are independent of
$h$, of the time step $\tau$ and of the time level $n$. For $v\in H^s(\domain)\cap H_0^1(\domain)$ with
$1 \leq s \leq p+1$ and $\Pi \in \{ \ltwoprojection, \ritzprojection, \interpolation \}$,
\begin{equation}\label{eq:fully_discrete_approximation}
  \| v - \Pi v \| + h \, \| \nabla ( v - \Pi v ) \|
  \lesssim
  h^{s} \, | v |_{H^s(\domain)},
  \qquad\text{and}\qquad
  \| \L^{\frac14} ( v - \Pi v ) \|
  \lesssim
  h^{s - \frac12} \, | v |_{H^{s}(\domain)} .
\end{equation}
The analysis further compares discrete
fractional norms of projected functions with continuous ones. For $v \in D ( \L^{a} )$,
\begin{alignat}{2}
  \| \discreteL^{a} \ltwoprojection v \| + \| \discreteL^{a} \interpolation v \|
  &\lesssim \| \L^{a} v \| ,
  &\qquad& 0 \leq a \leq \tfrac12 ,
  \label{eq:fully_discrete_fractional_stability}
  \\
  \| \discreteL^{a} \ritzprojection v \|
  &\leq \| \L^{a} v \| ,
  &\qquad& \tfrac12 \leq a \leq 1 .
  \label{eq:fully_discrete_fractional_stability_ritz}
\end{alignat}
\noindent The Ritz projection is absent
from~\eqref{eq:fully_discrete_fractional_stability} because it is defined only on
$H_0^1(\domain) = D ( \L^{\frac12} )$, which does not contain $D ( \L^{a} )$ for
$a < \tfrac12$. Where a discrete fractional norm of $\ritzprojection v$ below $a = \tfrac12$ is
needed, we use instead the uniform spectral bounds, valid for $\varphi_h \in V_h$, and~$0 \leq a \leq b \leq 1$,
\begin{align}
  \| \discreteL^{a} \varphi_h \|
  &\leq \lambda_1^{a-b} \, \| \discreteL^{b} \varphi_h \| ,  
  \label{eq:discrete_poincare}
  \\
  \| \discreteL^{b} \varphi_h \|
  &\leq \cstinv^{2(b-a)} h^{-2(b-a)} \, \| \discreteL^{a} \varphi_h \| ,
  \label{eq:discrete_inverse}
\end{align}
\noindent where $\lambda_1 > 0$ is the smallest eigenvalue of $\L$ and $\cstinv$ is the constant
of the inverse inequality $\| \nabla \varphi_h \| \leq \cstinv h^{-1} \| \varphi_h \|$, which
depends only on the shape-regularity and quasi-uniformity of the mesh family and on $p$. %
From the estimate~\eqref{eq:discrete_poincare} with $0\leq a \leq b=\frac12$ follows using~\eqref{eq:fully_discrete_fractional_stability_ritz} for $v \in H_0^1(\domain)$ that %
\begin{equation}\label{eq:fully_discrete_fractional_stability_ritz_below_half}
    \| \discreteL^{a} \ritzprojection v \|
  \leq \lambda_1^{a - \frac12} \| \L^{\frac12} v \| ,
\end{equation}
which replaces
\eqref{eq:fully_discrete_fractional_stability_ritz} below $a = \tfrac12$.

Finally, we present a lemma that will be used frequently in the analysis below. The proof is moved to \Cref{app:three_layer} as it consists mostly of standard ideas. Let for any fully discrete sequence~$v_h^{m+1},v_h^{m},v_h^{m-1}\in V_h$ and $\theta \in [\frac14,\frac12]$ the weighted average be defined as 
\begin{equation*}
  v_h^{m+1;\theta} \coloneqq \theta v_h^{m+1} + (1-2\theta) v_h^{m} + \theta v_h^{m-1}.
\end{equation*}

\begin{lemma}[Coercivity of the displacement energy]\label[lemma]{lem:three_layer_displacement_energy}
    Let $\theta \in [\frac14,\frac12]$, $0\leq a \leq \frac12$, and $\tau>0$ be some time step size. Define for any sequence~$u_h^{m+1},u_h^{m},u_h^{m-1}\in V_h$ the functions $v_h^k \coloneqq \discreteL^{a+\frac12} u_h^k$ for $k=m+1,m,m-1$. Then for $\mathcal{W}_h^m = \frac14 \|v_h^m + v_h^{m-1} \|^2 + (\theta - \frac14) \| v_h^m - v_h^{m-1} \|^2$ we have that $\mathcal{W}_h^m \geq 0$, {and it holds}
    \begin{alignat}{2}
        \langle
            v_h^{m+1;\theta},
            v_h^{m+1} - v_h^{m-1}
        \rangle
        &=
        \mathcal{W}_h^{m+1} - \mathcal{W}_h^{m},\label{eq:three_layer_alt}
        \\
        \mathcal{W}_h^m 
        &\leq
        \tfrac12 (\| v_h^m \|^2 + \| v_h^{m-1} \|^2),\label{eq:three_layer_upper}
        \\
        \| v_h^m \|^2
        &{}\leq \tfrac{2}{4\theta-1} \, \mathcal{W}_h^m ,
        &&\qquad  \theta > \tfrac14 ,\label{eq:three_layer_lower}
        \\
        \| v_h^m \|^2
        &{}\leq 2 \mathcal{W}_h^m + \tfrac{\cstinv^2}{2} {\tau^2}{h^{-2}} \| \discreteL^{a} \partial_\tau u_h^{m} \|^2 ,
        &&\qquad  \theta \geq \tfrac14 .\label{eq:three_layer_endpoint}
    \end{alignat}
\end{lemma}

\noindent Two consequences of \Cref{lem:three_layer_displacement_energy} are used so often
that we record them separately. They are the two one-sided estimates the energy argument needs,
and they are not an equivalence. At $\theta = \tfrac14$ the displacement energy
$\mathcal{W}_h^m$ alone does not control $\| v_h^m \|^2$, by the sawtooth
of~\Cref{rem:three_layer_sharpness}, and it is the combined displacement-plus-kinetic quantity
$\mathcal{F}_h^m$ {defined below} that does. %

\begin{corollary}[Energy control]\label[corollary]{cor:three_layer_energy_control}
  In the setting of \Cref{lem:three_layer_displacement_energy}, put
  $\mathcal{F}_h^m \coloneqq \| \discreteL^{a} \partial_\tau u_h^{m} \|^2 + \mathcal{W}_h^m$.
  Then
  \begin{alignat}{2}
    \tfrac14 \| \discreteL^{a+\frac12} ( u_h^{m} + u_h^{m-1} ) \|^2
    &\leq \mathcal{W}_h^m ,
    &\qquad& \theta \geq \tfrac14 ,
    \label{eq:energy_sum}
    \\
    \| \discreteL^{a+\frac12} u_h^{m} \|^2
    &\leq \cstcoerc \, \mathcal{F}_h^m ,
    &\qquad&
    \cstcoerc =
    \begin{cases}
      \tfrac{2}{4\theta-1}, & \theta > \tfrac14 , \\[2pt]
      \max \bigl\{ 2, \tfrac{\cstinv^2 \kappa^2}{2} \bigr\}, & \theta = \tfrac14 ,
    \end{cases}
    \label{eq:energy_coercive}
  \end{alignat}
  the second at $\theta = \tfrac14$ under the mesh condition $\tau \leq \kappa h$.
\end{corollary}

\begin{proof}
  \eqref{eq:energy_sum} is the definition of $\mathcal{W}_h^m$ together with
  $\theta \geq \tfrac14$, and \eqref{eq:energy_coercive} is \eqref{eq:three_layer_lower} for
  $\theta > \tfrac14$, respectively \eqref{eq:three_layer_endpoint} with
  $\tau^2 h^{-2} \leq \kappa^2$ at $\theta = \tfrac14$.
\end{proof}

\subsection{Fully discrete scheme and energy conservation}
\noindent Now we may formulate the \emph{fully discrete nonlinear scheme}: We rewrite the nonlinearity as follows 
\begin{equation*}
  \alpha^n + \beta^n 
  \bigl(
    \gamma \| \nabla v \|^2
    + (1 - 2\gamma) \| \nabla u_h^{n} \|^2
    + \gamma \| \nabla u_h^{n-1} \|^2
  \bigr)
  \eqqcolon
  \Gamma_h^{n;\gamma}(v).
\end{equation*}
Given initial conditions~$u_h^0,u_h^1\in V_h$, we seek a sequence of functions~$\{u_h^{n+1}\}_{n=1}^{N-1} \subset V_h$ that solves the following nonlinear equation 
\begin{equation}\label{eq:fully_discrete_nonlinear_scheme}
  \partial_\tau^2 u_h^{n+1}
  +
  \Gamma_h^{n;\gamma}(u_h^{n+1}) \discreteL u_h^{n+1;\theta}
  =
  f_h^{n+1;\eta},
\end{equation}
where the weighted average~$u_h^{n+1;\theta}$ is given as above, and analogously $f_h^{n+1;\eta} \coloneqq \eta f_h^{n+1} + (1-2\eta) f_h^{n} + \eta f_h^{n-1}$, with~$\gamma,\eta \in [0,\frac12]$ and~$\theta\in[\frac14,\frac12]$, and~$f_h^n = \interpolation f(n\tau)$ with an interpolation operator~$\interpolation$. We may use for the lowest order case~$p=1$ a Cl\'ement-type quasi-interpolation operator, and for the higher-order cases an Ern--Guermond-type quasi-interpolation operator. %
Note that for $\gamma=0$, the scheme reduces to the locally linear scheme from~\cite{KirchhoffTime}. In similar fashion, the scheme in \cite{KirchhoffTime} may also be extended by considering parameters $\gamma, \eta$, and $\theta$ as above.

Before turning to the convergence analysis we record that the scheme enjoys the following conservation property.

\begin{proposition}[Fully discrete energy conservation]\label[proposition]{prop:fully_discrete_conservation}
  Let $\theta = \tfrac12$, $\alpha$ and $\beta$ be constant, and $f \equiv 0$. Further, let $\{ u_h^{n} \}_{n=0}^{N} \subset V_h$ solve the
  nonlinear scheme~\eqref{eq:fully_discrete_nonlinear_scheme} exactly, and let $V_h$ be any
  conforming subspace of $H_0^1(\domain)$. Define the fully discrete energy 
\begin{align}\label{eq:energy-gamma}
  \mathcal{E}_{h}^{n;\gamma} &\coloneqq \tfrac12\bigl\|\partial_\tau u_h^n\bigr\|^2
   + \tfrac{\alpha}{4}\bigl(\bigl\|\nabla u_h^{n-1}\bigr\|^2+\bigl\|\nabla u_h^{n}\bigr\|^2\bigr)\nonumber \\
   &+ \tfrac{\beta}{4}\bigl[\bigl\|\nabla u_h^{n-1}\bigr\|^2\bigl\|\nabla u_h^{n}\bigr\|^2
     + \gamma\bigl(\bigl\|\nabla u_h^{n}\bigr\|^2-\bigl\|\nabla u_h^{n-1}\bigr\|^2\bigr)^{2}\bigr].
\end{align}
Then $\mathcal{E}_h^{n+1;\gamma}=\mathcal{E}_h^{n;\gamma}$ for $n=1,\dots,N-1$, and
consequently $\mathcal{E}_h^{n;\gamma}=\mathcal{E}_h^{1;\gamma}$ for $n=1,\dots,N$.
\end{proposition}

\begin{proof}
With $\theta=\tfrac12$ we have $u_h^{n+1;\theta}=\tfrac12\bigl(u_h^{n+1}+u_h^{n-1}\bigr)$. %
Taking $u_h^{n+1}-u_h^{n-1}$ as a test function in \eqref{eq:fully_discrete_nonlinear_scheme} gives
\begin{equation}
\label{eq:tested}
  \bigl\langle \partial_\tau^2 u_h^{n+1},\,u_h^{n+1}-u_h^{n-1}\bigr\rangle
  + \tfrac{\Gamma_h^{n;\gamma}(u_h^{n+1})}{2}\bigl\langle \discreteL\bigl(u_h^{n+1}+u_h^{n-1}\bigr),\,
      u_h^{n+1}-u_h^{n-1}\bigr\rangle = 0 .
\end{equation}
Throughout we abbreviate $\mathcal{V}_n\eqqcolon \tfrac12\|\nabla u_h^n\|^2$ and
$\mathcal{K}_n\eqqcolon \|\partial_\tau u_h^n\|^2$. For the first term in~\eqref{eq:tested}, write $u_h^{n+1}-2u_h^{n}+u_h^{n-1}
=\tau\bigl(\partial_\tau u_h^{n+1}-\partial_\tau u_h^{n}\bigr)$ and
$u_h^{n+1}-u_h^{n-1}=\tau\bigl(\partial_\tau u_h^{n+1}+\partial_\tau u_h^{n}\bigr)$,
so that it equals $\mathcal{K}_{n+1}-\mathcal{K}_n$. For the second, both arguments lie in $V_h$, so the
definition \eqref{eq:discrete_laplacian} of $\discreteL$ applies and turns it into
$\bigl\langle\nabla(u_h^{n+1}+u_h^{n-1}),\nabla(u_h^{n+1}-u_h^{n-1})\bigr\rangle
=\|\nabla u_h^{n+1}\|^2-\|\nabla u_h^{n-1}\|^2=2(\mathcal{V}_{n+1}-\mathcal{V}_{n-1})$.
Hence \eqref{eq:tested} becomes
\begin{equation}
\label{eq:tested-V}
  \mathcal{K}_{n+1}-\mathcal{K}_{n}+\Gamma_h^{n;\gamma}(u_h^{n+1})\bigl(\mathcal{V}_{n+1}-\mathcal{V}_{n-1}\bigr)=0 .
\end{equation}
Inserting $\Gamma_h^{n;\gamma}(u_h^{n+1})
  = \alpha+2\beta(\gamma\,\mathcal{V}_{n+1}+(1-2\gamma)\,\mathcal{V}_{n}+\gamma\,\mathcal{V}_{n-1})$ into
\eqref{eq:tested-V} therefore yields
\[
  \bigl(\mathcal{K}_{n+1}-\mathcal{K}_{n}\bigr)
  + \alpha\bigl(\mathcal{V}_{n+1}-\mathcal{V}_{n-1}\bigr)
  + 2\beta\gamma\bigl(\mathcal{V}_{n+1}^2-\mathcal{V}_{n-1}^2\bigr)
  + 2\beta(1-2\gamma)\mathcal{V}_n\bigl(\mathcal{V}_{n+1} - \mathcal{V}_{n-1}\bigr) = 0 .
\]
Dividing by $2$ and adding $\frac\alpha2\mathcal{V}_n$ and $\beta\gamma\mathcal{V}_n^2$ on both sides yields 
\begin{multline}\label{eq:energy_gamma_alt}
  \tfrac12 \mathcal{K}_{n+1}
  + \tfrac\alpha2\bigl(\mathcal{V}_{n+1}+\mathcal{V}_{n}\bigr)
  + \beta\gamma\bigl(\mathcal{V}_{n+1}^2+\mathcal{V}_{n}^2\bigr)
  + \beta(1-2\gamma)\mathcal{V}_{n+1}\mathcal{V}_{n}
  \\
  = \tfrac12 \mathcal{K}_{n}
  + \tfrac\alpha2\bigl(\mathcal{V}_{n}+\mathcal{V}_{n-1}\bigr)
  + \beta\gamma\bigl(\mathcal{V}_{n}^2+\mathcal{V}_{n-1}^2\bigr)
  + \beta(1-2\gamma)\mathcal{V}_{n}\mathcal{V}_{n-1}.
\end{multline}
Applying the elementary identity $\gamma\bigl(a^{2}+b^{2}\bigr)+(1-2\gamma)ab = ab+\gamma(a-b)^{2}$ for $a,b\in\mathbb{R}$ to the last two terms on both sides, then substituting back $\mathcal{V}_\cdot$ yields the energy identity $\mathcal{E}_h^{n+1;\gamma}=\mathcal{E}_h^{n;\gamma}$.
\end{proof}

\begin{remark}[Scope of \Cref{prop:fully_discrete_conservation}]\label[remark]{rem:conservation_scope}
  The identity holds for the exact solution of the nonlinear
  scheme~\eqref{eq:fully_discrete_nonlinear_scheme}. Any approximation of it such as the Picard
  iterates defined below satisfies it only up to the
  nonlinear residual. So an implementation conserves $\mathcal{E}_h^{n}$ only to the tolerance at which the nonlinear system is solved and, in floating point, to the accuracy of the linear solves. \\
    For $\beta\ge 0$ and any $\gamma\in[0,\tfrac12]$ every term of \eqref{eq:energy-gamma} is
    nonnegative, and by \eqref{eq:energy_gamma_alt},
    \[
      \mathcal{E}_h^{n;\gamma}\ \ge\
      \tfrac12\|\partial_\tau u_h^n\|^2
      + \tfrac{\alpha}{4}(\|\nabla u_h^{n-1}\|^2+\|\nabla u_h^{n}\|^2)
      + \tfrac{\beta\gamma}{4}(\|\nabla u_h^{n-1}\|^4+\|\nabla u_h^{n}\|^4),
    \]
    so that \Cref{prop:fully_discrete_conservation} yields bounds on
    $\|\partial_\tau u_h^n\|$ and $\|\nabla u_h^n\|$ that are uniform in $n$ and independent of
    $\tau$ and $h$. The quartic bound degenerates as $\gamma\to 0$, however, for $\alpha>0$ this is
    immaterial, since the $\|\nabla u_h^n\|^2$ terms already provide the required control. %
    
      The hypothesis $\theta=\tfrac12$ enters only through
  $u_h^{n+1;\theta} = \tfrac12 ( u_h^{n+1} + u_h^{n-1} )
   + ( \theta - \tfrac12 ) \tau^2 \partial_\tau^2 u_h^{n+1}$.
  Testing~\eqref{eq:fully_discrete_nonlinear_scheme} with $u_h^{n+1}-u_h^{n-1}$ as in the proof
  above, the second summand contributes
  $( \theta - \tfrac12 ) \tau^2 \Gamma_h^{n;\gamma}(u_h^{n+1}) ( D_{n+1} - D_{n} )$ with
  $D_m \coloneqq \| \nabla \partial_\tau u_h^{m} \|^2$, so that
  \begin{equation}\label{eq:energy_theta_defect}
    \mathcal{E}_h^{n+1;\gamma} - \mathcal{E}_h^{n;\gamma}
    =
    \tfrac12 \bigl( \tfrac12 - \theta \bigr) \tau^2 \,
    \Gamma_h^{n;\gamma}(u_h^{n+1}) \, \bigl( D_{n+1} - D_{n} \bigr) .
  \end{equation}
  This additional term does not telescope, because the coefficient $\Gamma_h^{n;\gamma}$ depends on the
  solution and hence on $n$. Summation by parts leaves
  $\Gamma_h^{n;\gamma} D_{n+1} - \Gamma_h^{1;\gamma} D_{1}
   - \sum_{m=2}^{n} ( \Gamma_h^{m;\gamma} - \Gamma_h^{m-1;\gamma} ) D_m$, whose last sum is
  $\bigO ( \tau \cdot n ) = \bigO ( T )$ under the assumption by the increment bound
  $| \Gamma_h^{m;\gamma} - \Gamma_h^{m-1;\gamma} | \lesssim \tau$
  of~\eqref{eq:cons_gamma}. %
    Note that this uses bounds that are not available at this point. We assume that
  $\Gamma_h^{m;\gamma}(u_h^{m+1})$ and $D_m$ are bounded uniformly in $m$, and that the
  increment bound above holds. The first and the third are of the kind established
  in~\Cref{sec:analysis} for the discrete solution, and the second is a half order stronger than
  the bounds used there, since $D_m = \| \discreteL^{\frac12} \partial_\tau u_h^m \|^2$. The
  estimate is therefore formal, and we do not use it anywhere below. %
  Hence, for $\theta \in [ \tfrac14, \tfrac12 )$ the discrete energy
  is not conserved but satisfies
  \begin{equation}\label{eq:energy_theta_order}
    \max_{1 \leq n \leq N} \bigl| \mathcal{E}_h^{n;\gamma} - \mathcal{E}_h^{1;\gamma} \bigr|
    \lesssim
    \bigl( \tfrac12 - \theta \bigr) \tau^2 ,
  \end{equation}
  that is, it is conserved to second order in $\tau$ with a constant growing linearly in $T$. This is confirmed
  numerically in~\Cref{sec:numerics}.
\end{remark}

\begin{remark}[The cases $\gamma=0$ and $\gamma=\frac12$]
\label[remark]{rem:gamma-half}
  The invariant $\mathcal{E}_h^{n;\gamma}$ is the fully discrete analogue of $\widetilde{E}_{1,n}$, and $\widetilde{E}_{2,n}$
  in~\cite{KirchhoffTime}. That is, for $\gamma\in\{0,\frac12\}$ the equations coincide up to the spatial discretization. Thus, we have that the energy $\mathcal{E}_h^{n;\gamma}$ does not conserve the continuous energy $E(t_n)$. %
For $\gamma=0$ the coefficient $\Gamma_h^{n;0}(u_h^{n+1})$ does not involve $u_h^{n+1}$, so that \eqref{eq:fully_discrete_nonlinear_scheme} is linearly implicit and Picard iteration is not needed. Nonetheless, \Cref{prop:fully_discrete_conservation} shows that this energy is conserved. 
  For general $\gamma$, comparing with the case $\gamma = \tfrac12$ gives 
\begin{equation}
\label{eq:energy-difference}
  \mathcal{E}_{h}^{n;\frac12}-\mathcal{E}_h^{n;\gamma}
  = \tfrac{\beta(1-2\gamma)}{8}
    (\|\nabla u_h^{n}\|^2-\|\nabla u_h^{n-1}\|^2)^{2}.
\end{equation}%
Since
$\|\nabla u_h^{n}\|^2-\|\nabla u_h^{n-1}\|^2
=\tau\bigl\langle\nabla(u_h^{n}+u_h^{n-1}),\nabla\partial_\tau u_h^{n}\bigr\rangle$,
the difference \eqref{eq:energy-difference} is of order $\tau^{2}$. Moreover, it is itself a
function of the two levels $n-1,n$ only, whence
\[
  \mathcal{E}_{h}^{n;\frac12}
  = \mathcal{E}_{h}^{1;\frac12}
    + \tfrac{\beta(1-2\gamma)}{8}[
      (\|\nabla u_h^{n}\|^2-\|\nabla u_h^{n-1}\|^2)^{2}
      -(\|\nabla u_h^{1}\|^2-\|\nabla u_h^{0}\|^2)^{2}],
  \qquad n=1,\dots,N.
\]
Thus, for $\gamma\neq\tfrac12$ the quantity $\mathcal{E}_{h}^{n;\frac12}$ is not conserved, but its
deviation is bounded by $\bigO(\tau^{2})$ uniformly in $n$ and does not accumulate over
$[0,T]$.
\end{remark}

\section{Analysis of the fully discrete scheme}\label{sec:analysis}

\noindent To solve the nonlinear scheme \eqref{eq:fully_discrete_nonlinear_scheme} at every time step, we apply a Picard iteration that leads to solving a linear equation for every iterate. That is, we seek a sequence of functions~$\{u_{h;\ell +1}^{n+1}\}_{\ell \geq 0} \subset V_h$ that solves the \emph{fully discrete linear iteration scheme} 
\begin{equation}\label{eq:fully_discrete_nonlinear_iteration_scheme}
  u_{h;\ell +1}^{n+1} + \tau^2 \theta \Gamma_{h}^{n;\gamma}(u_{h;\ell }^{n+1}) \discreteL u_{h;\ell +1}^{n+1}
  = \tau^2 f_h^{n+1;\eta} + 2u_h^n - u_h^{n-1}
  - \tau^2 \Gamma_{h}^{n;\gamma}(u_{h;\ell }^{n+1}) \discreteL ((1 - 2\theta) u_h^n + \theta u_{h}^{n-1})
\end{equation}
with the initial guess~$u_{h;0}^{n+1} = u_h^n$. %
In this section we will prove that this scheme admits a unique solution trajectory over the whole time interval $[0,T]$ where the exact solution is bounded. This is based on a contraction principle (fixed point argument) of the Picard iteration with the following linear operator %
\begin{equation*}
  \T_{v} = \mathcal{I} + \tau^2 \theta \Gamma_h^{n;\gamma}(v) \discreteL.
\end{equation*}
Then, the Picard iteration seeks a fixed point of (note that $\theta \geq \frac14$)
\begin{equation}\label{eq:fully_discrete_fixed_point_iteration}
  v = F(v) = \T_v^{-1} g_h^{n} + (2 - \tfrac1\theta) u_h^n - u_h^{n-1},
  \qquad
  g_h^n
  \coloneqq
  \tau^2 f_h^{n+1;\eta}
  +
  \tfrac1\theta u_h^n.
\end{equation}

\subsection{Local existence of a solution}
\noindent As a first step we show that the nonlinear equation~\eqref{eq:fully_discrete_nonlinear_scheme} has a unique solution at each time level, and that the Picard iteration converges to it. Both follow from a contraction principle applied to the fixed point map~\eqref{eq:fully_discrete_fixed_point_iteration}, whose fixed point is that solution. This is formulated in the following lemma. %
\begin{lemma}[Existence of a solution of the nonlinear discrete equation]\label[lemma]{lem:fully_discrete_local_existence}
  Let~\Cref{ass:existence_exact_solution} hold. Further, let the fully discrete solutions be bounded by 
  \begin{equation}\label{eq:fully_discrete_bounds_previous_time_steps}
    \begin{alignedat}{2}
      \| \discreteL^\frac34 u_h^n \|_{L^2(\domain)} &\leq M_\frac34,
      \qquad
      \| \discreteL^\frac14 u_h^n \|_{L^2(\domain)} &&\leq M_\frac14,
      \qquad
      \| \discreteL^\frac14 \partial_\tau u_h^n \|_{L^2(\domain)} \leq M'_\frac14,
      \\
      \| \discreteL^\frac34 u_h^{n-1} \|_{L^2(\domain)} &\leq M_\frac34,
      \qquad
      \| \discreteL^\frac14 u_h^{n-1} \|_{L^2(\domain)} &&\leq M_\frac14
    \end{alignedat}
  \end{equation}
  for some~$1\leq n \leq N-1$. 
  Then there exists $\tau_0>0$, depending only on $\alpha_{\min}$, $\beta_{\max}$, $\gamma$, $\theta$,
  the final time~$T$ and the constants $M_{\frac14}, M_{\frac34}, C_{\frac14}$ of
  \Cref{ass:existence_exact_solution}, but \emph{not} on the time step $n$, the time step
  size $\tau$ or the mesh size $h$, such that for every $0 < \tau \leq \tau_0$
  the nonlinear equation~\eqref{eq:fully_discrete_nonlinear_scheme} has a unique solution~$u_h^{n+1}\in V_h$, to which the Picard iteration~\eqref{eq:fully_discrete_nonlinear_iteration_scheme} converges. Further, for the following norms we have that 
  \begin{equation*}
    \| \discreteL^\frac34 u_h^{n+1} \|_{L^2(\domain)} \lesssim 1,
    \qquad
    \| \discreteL^\frac14 u_h^{n+1} \|_{L^2(\domain)} \lesssim 1,
    \qquad
    \| \discreteL^\frac14 \partial_\tau u_h^{n+1} \|_{L^2(\domain)} \lesssim 1
  \end{equation*}
  hold independent of the time step~$n$, the time step size~$\tau$, and the mesh size~$h$.
\end{lemma}

\begin{proof}
  We show this lemma using a contraction principle on the mapping~$F$ introduced in~\eqref{eq:fully_discrete_fixed_point_iteration}. %
  We first show that $F$ is a self-map on a set bounded in both the $\discreteL^\frac34$- and
  the $\discreteL^\frac14$-norm, with bounds independent of the preimage: the contraction is
  performed in the $\discreteL^\frac14$-norm, and the $\discreteL^\frac34$-bound is what the
  error estimate below needs. Let~$a \in \{ \frac14,\frac34\}$. Using~$\frac1\theta \geq 2$, we have for any~$\varphi_h\in V_h$ that 
  \begin{align*}
    &\| \discreteL^a  F(\varphi_h)  \| 
    =
    \|
    \discreteL^a  \T_{\varphi_h}^{-1}
    \bigl( 
      \tau^2 f_h^{n+1;\eta}
      + \tfrac1\theta u_h^n 
    \bigr)
    +
    \discreteL^a
    ((2 - \tfrac1\theta) u_h^n - u_h^{n-1})  \| 
    \\
      &\qquad\leq 
      \tau^{-2(\frac14 - a)}  \| \discreteL^{a-\frac14} \T_{\varphi_h}^{-1} \|
      \| \tau^{2(\frac54 - a)} \discreteL^{\frac14} f_h^{n+1;\eta} \|
      +
      \| \T_{\varphi_h}^{-1} \|
      \| \tfrac1\theta \discreteL^a u_h^n \|
      +
      |2 - \tfrac1\theta| \, \|  \discreteL^a u_h^n   \|
      +
      \|\discreteL^a   u_h^{n-1} \|
        \\
      &\qquad\leq 
      \tau^{-2(\frac14 - a)} \tau^{2(\frac14 - a)} c(a - \tfrac14) 
      \| \tau^{2(\frac54 - a)} \discreteL^{\frac14} f_h^{n+1;\eta} \|
      +
      (\tfrac2\theta - 2) \|  \discreteL^a u_h^n   \|
      +
      \|\discreteL^a   u_h^{n-1} \|
        \\
      &\qquad\leq 
      c(a - \tfrac14) \tau^{2(\frac54 - a)} 
      \| \discreteL^{\frac14} f_h^{n+1;\eta} \|
      +
      (\tfrac2\theta - 2) \|  \discreteL^a u_h^n   \|
      +
      \|\discreteL^a   u_h^{n-1} \|,
  \end{align*}
  where for~$0 \leq a \leq 1$ we have that $
    c(0) = 1,
    c(a)
    =
    \bigl(
      \tfrac{a}{\theta \alpha_{\min}}
    \bigr)^{a} {(1-a)}^{1-a},
    c(1) = \tfrac{1}{\theta\alpha_{\min}}.$
  Note that the source term mean needs separate handling as we do not invoke the~$\discreteL^\frac34$-norm. %
  Thus, for~$a=\frac14$ and~$a=\frac34$, respectively, using \eqref{eq:fully_discrete_fractional_stability} and \eqref{eq:fully_discrete_fractional_stability_ritz} we obtain 
  \begin{alignat*}{2}
    \| \discreteL^{\frac14}  F(\varphi_h)  \|
    &\leq
    \tau^2 C_{\frac14} + (\tfrac2\theta - 1) M_\frac14
    \leq
    T^2 C_{\frac14} + (\tfrac2\theta - 1) M_\frac14
    &&\eqqcolon
    K_\frac14,
    \\
    \| \discreteL^{\frac34}  F(\varphi_h)  \|
    &\leq
    \tfrac1{2\sqrt{\theta\alpha_{\min}}} \tau C_{\frac14} + (\tfrac2\theta - 1) M_\frac34
    \leq
    \tfrac1{2\sqrt{\theta\alpha_{\min}}} T C_{\frac14} + (\tfrac2\theta - 1) M_\frac34
    &&\eqqcolon
    K_\frac34
  \end{alignat*}
  without any \apriori{} bounds on $\varphi_h$, specifically, this means that the constants~$K_a$ for~$a\in\{\frac14,\frac34\}$ do not depend on the preimage. %
  We construct a complete subset of the finite element space~$V_h$ to be 
  \begin{equation*}
    K \coloneqq
    \bigl\{ \varphi_h \in V_h \mid 
    \|\discreteL^a \varphi_h \| \leq K_a \textrm{ for } a\in \{\tfrac14,\tfrac34\} \bigr\},
  \end{equation*}
  such that the fixed point map~$F$ is a self map on~$K$. %
  The contraction is performed in the~$\discreteL^\frac14$-norm. The reason is the structure of
  the nonlinearity: since~$\|\nabla \cdot\| = \| \discreteL^\frac12 \cdot\|$, its increment is
  proportional to the increment of its argument in the~$\discreteL^a$-norm for any
  $0\leq a \leq 1$,
  \begin{equation}\label{eq:fully_discrete_nonlinearity_proportional}
    \begin{aligned}
      \left|
        \Gamma_h^{n;\gamma}(u) - \Gamma_h^{n;\gamma}(v)
      \right|
      &=
      \beta^n \gamma
      \big| \|\nabla u\|^2 - \| \nabla v\|^2\big| 
      =
      \beta^n\gamma
      |\langle
        \nabla (u - v) , \nabla (u+v)
      \rangle |
      \\
      &\leq
      \beta^n \gamma
      \bigl(
        \|\discreteL^{1-a} u\| + \|\discreteL^{1-a} v\|
      \bigr)
      \left\|
        \discreteL^a (u - v)
      \right\|.
    \end{aligned}
  \end{equation}
  At $a=0$ this would require $H^2$-regularity, through
  $\|\discreteL u_h^n\|$. Following~\cite{wellposedness_kirchhoff_1996}, the choice
  $a=\frac14$ avoids it, and it is the one made here. Using
  the definition~\eqref{eq:fully_discrete_fixed_point_iteration} and the fact that~$\T_v^{-1}$
  and~$\discreteL$ commute, we have for any~$\varphi_h,\psi_h\in V_h$ that 
  \begin{equation}\label{eq:fully_discrete_contraction}
    \begin{aligned}
    \| \discreteL^a (F(\varphi_h) -  F(\psi_h) )  \|
    &=
    \| \discreteL^a (  \T_{\varphi_h}^{-1} - \T_{\psi_h}^{-1}    ) g_h^n \|
    =
    \| \T_{\varphi_h}^{-1} (  \T_{\psi_h} - \T_{\varphi_h}    ) \discreteL^{-1}  \discreteL^{1+a}  \T_{\psi_h}^{-1}  g_h^n \|
    \\
    &\leq
    \| \T_{\varphi_h}^{-1} (  \T_{\varphi_h} - \T_{\psi_h}    ) \discreteL^{-1} \|\, \| \discreteL^{1+a}  \T_{\psi_h}^{-1}  g_h^n \|.
    \end{aligned}
  \end{equation}
  The norm of the first term can be estimated by writing out the difference, and using the proportionality argument~\eqref{eq:fully_discrete_nonlinearity_proportional}. We have using also~$\|\T_v^{-1}\|\leq 1$ that 
  \begin{align*}
    \| \T_{\varphi_h}^{-1} (  \T_{\varphi_h} - \T_{\psi_h}    ) \discreteL^{-1} \|
    &=
    |  \tau^2 \theta  ( \Gamma_h^{n;\gamma} (\varphi_h) - \Gamma_h^{n;\gamma} (\psi_h)) |
    \\
    &\leq
    \tau^2 \theta  \beta^n \gamma (  \| \discreteL^{1-a} \varphi_h  \| + \| \discreteL^{1-a} \psi_h  \|  ) 
    \| \discreteL^{a} (\varphi_h - \psi_h)  \|.
  \end{align*}
  Here, on the right-hand side we observe the difference of functions~$\varphi_h,\psi_h$ in the~$\discreteL^a$-norm, as on the left-hand side of~\eqref{eq:fully_discrete_contraction}. Thus, it is left to bound the last term on the right-hand side of~\eqref{eq:fully_discrete_contraction}. With similar arguments as above we have that 
  \begin{align*}
    \| \discreteL^{1+a}  \T_{\psi_h}^{-1}  g_h^n \|  
    &=
    \| \tau^2 \T_{\psi_h}^{-1} \discreteL \discreteL^a f_h^{n+1;\eta}
    +
    \tfrac1\theta \T_{\psi_h}^{-1} \discreteL^{1 - 2a} \discreteL^{3a} u_h^n \|
    \\
    &\leq
    \tau \| \T_{\psi_h}^{-1} \discreteL \|\, \| \tau \discreteL^a f_h^{n+1;\eta} \|
    +
    \tfrac1\theta \| \T_{\psi_h}^{-1} \discreteL^{1 - 2a} \|\, \| \discreteL^{3a} u_h^n \|
    \\
    &\leq
    \tau^{-1} c(1) \| \tau \discreteL^a f_h^{n+1;\eta} \|
    +
    \tau^{2(2a - 1)} \tfrac1\theta c(1-2a) \| \discreteL^{3a} u_h^n \|.
  \end{align*}
  Combining the two previous estimates into~\eqref{eq:fully_discrete_contraction}, setting~$a=\frac14$ and using the bounds in~\Cref{ass:existence_exact_solution} and in~\eqref{eq:fully_discrete_bounds_previous_time_steps} we have that 
  \begin{align*}
    &\| \discreteL^{\frac14} (F(\varphi_h) -  F(\psi_h) )  \|
    \\
    &\qquad\leq
      {\tau^2} {\theta\beta^n\gamma}
      \bigl(
        \| \discreteL^\frac34 \varphi_h \| + \| \discreteL^\frac34 \psi_h \|
      \bigr)
      \bigl(
        \tau^{-1} \tfrac1{\theta \alpha_{\min}} \| \tau \discreteL^{\frac14} f_h^{n+1;\eta} \|
        +
        \tau^{-1} \tfrac1\theta c(\tfrac12) \| \discreteL^{\frac34} u_h^n \|
      \bigr)
      \| \discreteL^\frac14 \varphi_h - \discreteL^{\frac14} \psi_h \|
      \\
      &\qquad\leq
      \tau
      {\beta^n\gamma}
      (  \| \discreteL^\frac34 \varphi_h  \| + \| \discreteL^\frac34 \psi_h  \|  ) 
      \bigl(
        \tau \tfrac1{\alpha_{\min}} C_{\frac14}
        +
        \tfrac1{2\sqrt{\theta\alpha_{\min}}} M_{\frac34} 
      \bigr)
      \| \discreteL^{\frac14} (\varphi_h - \psi_h)  \|.
  \end{align*}
  Finally, for any~$\varphi_h,\psi_h\in K$ we have that 
  \begin{align*}
    \| \discreteL^{\frac14} (F(\varphi_h) -  F(\psi_h) )  \|
    &\leq
    \tau \beta_{\max}\gamma
    2 K_\frac34 
    \bigl(
      T \tfrac1{\alpha_{\min}} C_{\frac14}
      +
      \tfrac1{2\sqrt{\theta\alpha_{\min}}} M_{\frac34} 
    \bigr)
    \| \discreteL^{\frac14} (\varphi_h - \psi_h)  \|.
  \end{align*}
  Thus, there exists~$\tau_0 > 0$ that may depend on the coefficients~$\alpha$ and~$\beta$, the final time~$T$, the parameters~$\gamma$ and~$\theta$, the norms of the exact solution~$u$ and the source term~$f$, but does not depend on the time step~$n$ and the mesh size~$h$ such that, for~$0<\tau\leq \tau_0$, the mapping~$F$ is a contractive self-mapping on the complete metric space~$K$. The bounds placing $F ( \varphi_h )$ in $K$ were obtained without any restriction on the preimage, so $F ( V_h ) \subset K$ and every fixed point of $F$ lies in $K$, uniqueness in $K$ is therefore uniqueness in $V_h$. By the contraction principle, there exists a unique fixed point~$u_h^{n+1} \in V_h$ with the following bounds 
  $
    \| \discreteL^\frac34 u_h^{n+1} \| \leq K_\frac34,
    \| \discreteL^\frac14 u_h^{n+1} \| \leq K_\frac14$.
  Lastly, $\| \discreteL^\frac14 \partial_\tau u_h^{n+1} \|$ is bounded by the discrete energy
  argument. Testing~\eqref{eq:fully_discrete_nonlinear_scheme} with
  $\frac12\discreteL^\frac12(\partial_\tau u_h^{n+1} + \partial_\tau u_h^n)$ and \eqref{eq:three_layer_alt} gives 
  \begin{multline*}
    \| \discreteL^\frac14 \partial_\tau u_h^{n+1} \|^2
    + \Gamma_h^{n;\gamma}(u_h^{n+1}) 
     \mathcal{W}_h^{n+1} 
    \\
    \leq
    \| \discreteL^\frac14 \partial_\tau u_h^{n} \|^2
    + \Gamma_h^{n;\gamma}(u_h^{n+1}) 
     \mathcal{W}_h^{n} 
    +
    \| \discreteL^\frac14 f_h^{n+1;\eta} \|
    \| \discreteL^\frac14 (u_h^{n+1} - u_h^{n-1}) \| ,
  \end{multline*}
  where~$\mathcal{W}_h^n$ as in \Cref{lem:three_layer_displacement_energy}. 
  Since $\theta\geq \frac14$ and $\Gamma_h^{n;\gamma}(u_h^{n+1}) \geq 0$, the left-hand side has the lower bound $\| \discreteL^\frac14 \partial_\tau u_h^{n+1} \|^2$. Further, by assumption \eqref{eq:fully_discrete_bounds_previous_time_steps} we have $\Gamma_h^{n;\gamma}(u_h^{n+1}) \leq \alpha_{\max} + \beta_{\max}(\gamma K_\frac14 K_\frac34 + (1 - \gamma) M_\frac14 M_\frac34)\eqqcolon C_\Gamma$, and using~\eqref{eq:three_layer_upper} we have that 
  \begin{equation*}
    \| \discreteL^\frac14 \partial_\tau u_h^{n+1} \|^2
    \leq
    {M'_\frac14}^2
    +
    C_\Gamma M_\frac34^2
    +
    C_\frac14 ( K_\frac14 + M_\frac14 ) ,
  \end{equation*}
  which is independent of $n$, $\tau$ and $h$.
\end{proof}

The previous lemma shows the existence of a solution $u_h^{n+1}$ for some~$1\leq n\leq N-1$, however, it has a crucial drawback. Since we are using a contraction argument, the norms of the solution are larger than the ones assumed at the previous time steps. A recursive application of this lemma would lead to a blow-up of the solution norms. The following sections provide an idea of how to avoid the blow-up and show that the solution exists for all time steps. 

\subsection{Convergence of the local solutions}
We avoid the issue of a potential blow-up by showing that the solution $u_h^{n+1}$ approximates the Ritz projection of the exact solution at time $t_{n+1}$, and, thus, lives in the same ball as assumed for the previous time steps. %
This leads to the fact that the norms of all discrete solutions are bounded and we may apply recursively the previous lemma. 
In the following we abbreviate~$\hat u^n = u(n\tau)$ for the exact solution. %
Throughout this section, in addition to the assumptions in \Cref{lem:fully_discrete_local_existence}, we assume that a whole trajectory of fully discrete solutions exists and satisfies norm and error bounds. This is summarized in the following hypothesis. %

\begin{hypothesis}[Induction hypothesis]\label[hypothesis]{hyp:fully_discrete_existence_trajectory}
  Let the fully discrete solution be uniformly bounded by 
  \begin{subequations}\label{eq:fully_discrete_bounds_all_time_steps}
  \begin{equation}
    \begin{alignedat}{2}
      \| \discreteL^\frac34 u_h^m \|_{L^2(\domain)} &\leq M_\frac34,
      \qquad
      \| \discreteL^\frac14 u_h^m \|_{L^2(\domain)} \leq M_\frac14,
      \qquad
      &&m=0,\dots,n,
      \\
      \| \discreteL^\frac14 \partial_\tau u_h^m \|_{L^2(\domain)} &\leq M'_\frac14,
      \qquad
      &&m=1,\dots,n
    \end{alignedat}
  \end{equation}
  for some~$1\leq n \leq N-1$, and assume the following error estimates 
  \begin{alignat}{2}
    \| \discreteL^{\frac12 + a} ( \ritzprojection u(t_m) - u_h^{m} ) \|_{L^2(\domain)}
    +
    \| \discreteL^a ( \ritzprojection u(t_m) - u_h^{m} ) \|_{L^2(\domain)} 
    &\lesssim
    h^{p+1-2a} + \tau^2,
    \qquad
    m=0,\dots,n,\notag
    \\
    \| \discreteL^a \partial_\tau ( \ritzprojection u(t_m) - u_h^{m} ) \|_{L^2(\domain)}
    &\lesssim
    h^{p+1-2a} + \tau^2,
    \qquad
    m=1,\dots,n
  \end{alignat}
  \end{subequations}
  hold for $a\in\{0,\frac14\}$, where the constants are independent of the time step~$n$, the time step size~$\tau$, and the mesh size~$h$.
\end{hypothesis}

Note that \Cref{hyp:fully_discrete_existence_trajectory} implies the
hypotheses~\eqref{eq:fully_discrete_bounds_previous_time_steps}
of~\Cref{lem:fully_discrete_local_existence}, so that a unique
solution~$u_h^{n+1} \in V_h$ of~\eqref{eq:fully_discrete_nonlinear_scheme} exists. %
Regarding the choice of the parameter~$a$, the hypothesis is stated for both $a=0$ and $a=\frac14$ for the reason that the existence of a trajectory requires the stronger norms with $a=\frac14$, and standard error estimates are performed in the weaker norms with $a=0$. %
To show the convergence consider the fully discrete scheme and insert the Ritz projection of the exact solution $\ritzprojection \hat u^n$ into the fully discrete Kirchhoff PDE~\eqref{eq:fully_discrete_nonlinear_scheme}, then multiplying with a test function~$\varphi_h\in V_h$ and integrating we obtain 
\begin{equation}\label{eq:fully_discrete_nonlinear_scheme_defects}
  \langle
    \partial_\tau^2 \ritzprojection \hat u^{n+1}, \varphi_h
  \rangle
  +
  \Gamma_h^{n;\gamma}(u_h^{n+1})
  \langle
    \nabla \ritzprojection \hat u^{n+1;\theta},\nabla \varphi_h
  \rangle
  =
  \langle
    f_h^{n+1;\eta}, \varphi_h
  \rangle
  +
  \langle
    \delta_{h;\frac14}^n + \delta_{h;\frac34}^n + S_h^n, \varphi_h 
  \rangle.
\end{equation}
Here,~$\delta_{h;\frac14}^n + \delta_{h;\frac34}^n = [\delta_{h;f}^n + \delta_{h;1}^n] + [\delta_{h;2}^n + \delta_{h;3}^n + \delta_{h;4}^n]$ denotes a sum of defects, where the split is according to the
analysis below, and~$S_h^n$ is a stabilization term that arises from the nonlinearity. %
Writing $q^m = \alpha^m + \beta^m \| \nabla \hat u^m \|^2$ for the coefficient of the continuous equation at $t_m$, applying $\ltwoprojection$ to that equation and using the commutation identity~\eqref{eq:ritz_commutation} in the form $\discreteL \ritzprojection \hat u^m = \ltwoprojection \L \hat u^m$, we have that 
\begin{align*}
  \delta_{h;f}^m &= \ltwoprojection f^m - f_h^{m+1;\eta},
  \qquad
  \delta_{h;1}^m = \partial_\tau^2 \ritzprojection \hat u^{m+1} - \ltwoprojection \partial_t^2 \hat u^m,
  \\
  \delta_{h;2}^m &= \Gamma_h^{m;\gamma}(u_h^{m+1})\, \discreteL \ritzprojection \theta \bigl( \hat u^{m+1} - 2\hat u^m + \hat u^{m-1} \bigr),
  \\
  \delta_{h;3}^m &= \beta^m \gamma \bigl( \| \nabla \hat u^{m+1} \|^2 - 2 \| \nabla \hat u^{m} \|^2 + \| \nabla \hat u^{m-1} \|^2 \bigr) \discreteL \ritzprojection \hat u^m,
\end{align*}
and, with the projection and error contributions to the coefficient $
  \rho^m = \| \nabla \ritzprojection \hat u^{m} \|^2 - \| \nabla \hat u^{m} \|^2,
  \sigma^m = \| \nabla u_h^{m} \|^2 - \| \nabla \ritzprojection \hat u^{m} \|^2$,
\begin{alignat*}{2}
  \delta_{h;4}^m &{}= \beta^m \bigl( \gamma \rho^{m+1} + (1-2\gamma)\rho^m + \gamma \rho^{m-1} \bigr) \discreteL \ritzprojection \hat u^m,
  \\
  S_h^m &{}= \beta^m \bigl( \gamma \sigma^{m+1} + (1-2\gamma)\sigma^m + \gamma \sigma^{m-1} \bigr) \discreteL \ritzprojection \hat u^m.
\end{alignat*}
The split is exactly the one obtained by writing
$\Gamma_h^{m;\gamma}(u_h^{m+1}) - q^m$ as a sum of a $\sigma$-part, a $\rho$-part and a purely
temporal part, as carried out below. %
From the definitions of the defects, we readily observe that the split of the defect are according to if they carry a discrete Laplacian or not. The reason for it is that the terms are handled differently in the analysis. %
And only $S_h^m$ is proportional to the coefficient \emph{error},
which is why it alone requires the stabilization argument. %
Before we prove the theorem, we first provide the estimates for the defects and stabilization terms such that the proof is concise and readable. The proof is given in \Cref{app:defects}.

\begin{lemma}[Consistency]\label[lemma]{lem:consistency}
  Let~\Cref{ass:existence_exact_solution,hyp:fully_discrete_existence_trajectory} hold and let~$a\in \{0,\frac14\}$. Further, assume the additional temporal regularity for the solution~$u$ to the PDE \eqref{eq:main_kirchhoff}
  \begin{alignat*}{2}
    \partial_t^j u &\in L^\infty\bigl(0,T;D ( \L^{\frac34} )\bigr),\qquad && j=1,2,3,
    \\
    \partial_t^4 u &\in L^\infty\bigl(0,T;D ( \L^{\frac14} )\bigr),
  \end{alignat*}
  as well as the spatial regularity
  \begin{alignat*}{2}
    \partial_t^j u &{}\in L^\infty(0,T;H^{p+1}(\domain)),\qquad && j=0,1,2,
    \\
    f &{}\in L^\infty\bigl(0,T;H^{p+1}(\domain) \cap H_0^1(\domain)\bigr)\,,\qquad && \partial_t^2 f \in L^\infty\bigl(0,T;D ( \L^{\frac14} )\bigr).
  \end{alignat*}
  Then, we have that
  \begin{alignat}{2}
      | \Gamma_h^{m;\gamma}(u_h^{m+1}) - \Gamma_h^{m-1;\gamma}(u_h^{m}) | &\lesssim \tau , &\qquad& m=2,\dots,n \label{eq:cons_gamma}
      \\
    \| \discreteL^{a} \delta_{h;\frac14}^m \| &\lesssim h^{p+1-2a} + \tau^2 , &\qquad& m=1,\dots,n
      \label{eq:cons_weak}
      \\
    \| \discreteL^{a - \frac12} \delta_{h;\frac34}^m \|
      &\lesssim h^{2p} + \tau^2 , &\qquad& m\in\{1,n\}
      \label{eq:cons_stiff_end}
      \\
    \| \discreteL^{a - \frac12} \partial_\tau \delta_{h;\frac34}^m \|
      &\lesssim h^{2p} + \tau^2 , &\qquad& m = 2, \dots, n
      \label{eq:cons_stiff_time}
      \\
    | \langle \discreteL^{a} S_h^m , \discreteL^{a} \varphi_h \rangle |
      &\lesssim \Bigl( \textstyle\sum_{k=m-1}^{m+1} \| \discreteL^{\frac12 + a} e_h^{k} \| \Bigr)
        \| \discreteL^{a} \varphi_h \| , &\qquad& m = 1, \dots, n,
      \label{eq:cons_stab}
  \end{alignat}
  where the constants are independent of the time steps~$n$ and~$m$, the time step size~$\tau$, and the mesh size~$h$.
\end{lemma}

Note that here, and only here, we assume the additional spatial regularity in the Sobolev scale, since these hypotheses feed the finite element approximation estimates. %
The requirement $f(t) \in H_0^1(\domain)$ makes the approximation
estimate~\eqref{eq:fully_discrete_approximation} applicable to $f$. It is an
{additional} hypothesis, not a consequence of \Cref{ass:existence_exact_solution}, whose
$D ( \L^{\frac14} ) = H_{00}^{\frac12}(\domain)$ does not embed in $H_0^1(\domain)$. %

Thus, we may give the main convergence theorem. The solution $u_h^{n+1}$ of
\Cref{lem:fully_discrete_local_existence} approximates $\ritzprojection u$ with order
$\bigO ( h^{p+\frac12} + \tau^2 )$ in the $\discreteL^{\frac34}$-norm, under stronger
regularity in space and time and given boundedness of the discrete solution at all previous
time steps. It is the induction step; the base case and the statement for all time levels are
\Cref{lem:fully_discrete_initial_conditions,cor:fully_discrete_convergence}. All estimates
below are for the error $\ritzprojection u(t_m) - u_h^{m}$ relative to the Ritz projection,
not for the full error: the projection error $u(t_m) - \ritzprojection u(t_m)$ is not a finite
element function and is therefore not measured in the discrete fractional norms at all. It is
added back, in the $L^2$- and $H^1$-norms, in \Cref{cor:fully_discrete_optimal_error}.

\begin{theorem}[Convergence of the local solution]\label{thm:fully_discrete_local_convergence}
  Let the assumptions of \Cref{lem:consistency} hold, let $\theta \in [ \tfrac14, \tfrac12 ]$,
  and assume at the endpoint $\theta = \tfrac14$ in addition the mesh condition
  $\tau \leq \kappa h$ for some fixed $\kappa > 0$.
  Then, the unique solution~$u_h^{n+1}\in V_h$ of the nonlinear equation~\eqref{eq:fully_discrete_nonlinear_scheme} supplied by~\Cref{lem:fully_discrete_local_existence} satisfies the same estimates at the level~$n+1$ as the solutions at the previous steps. That is, we have that 
  \begin{align*}
    \| \discreteL^{\frac12 + a} ( \ritzprojection u(t_{n+1}) - u_h^{n+1} ) \|_{L^2(\domain)}
    +
    \| \discreteL^a \partial_\tau ( \ritzprojection u(t_{n+1}) - u_h^{n+1} ) \|_{L^2(\domain)}
    &\lesssim
    h^{p+1-2a} + \tau^2,
    \\
    \| \discreteL^a ( \ritzprojection u(t_{n+1}) - u_h^{n+1} ) \|_{L^2(\domain)} 
    &\lesssim
    h^{p+1-2a} + \tau^2,
  \end{align*}
  where the constants are independent of the time step~$n$, the time step size~$\tau$, and the mesh size~$h$. 
\end{theorem}

\begin{proof}
First, we derive the final error equation such that we may apply the discrete Gr\"onwall inequality. 
Subtracting the fully discrete equation~\eqref{eq:fully_discrete_nonlinear_scheme} from~\eqref{eq:fully_discrete_nonlinear_scheme_defects}, we have that 
\begin{equation}\label{eq:error_equation}
  \langle
    \partial_\tau^2 e_h^{n+1}, \varphi_h
  \rangle
  +
  \Gamma_h^{n;\gamma}(u_h^{n+1})
  \langle
    \nabla e_h^{n+1;\theta},\nabla \varphi_h
  \rangle
  =
  \langle
    \delta_{h;\frac14}^n + \delta_{h;\frac34}^n + S_h^n, \varphi_h
  \rangle.
\end{equation}
Let $a\in \{ 0, \frac14\}$, and define the test function
\(
  \varphi_h = \tfrac12 \discreteL^{2a} (\partial_\tau e_h^{n+1} + \partial_\tau e_h^n)
  \in V_h ,
\)
and define the three layer displacement energy by 
\(
  \mathcal{W}_h^m
  \coloneqq
  \tfrac14 \| \discreteL^{\frac12 + a} (e_h^{m} + e_h^{m-1}) \|^2 + \bigl( \theta - \tfrac14 \bigr) \| \discreteL^{\frac12 + a} (e_h^{m} - e_h^{m-1}) \|^2 .
\)
Note that this definition coincides with the definition in \Cref{lem:three_layer_displacement_energy}. 
Using 
$\tfrac12 ( \partial_\tau e_h^{n+1} + \partial_\tau e_h^{n} )
= \frac1{2\tau} ( e_h^{n+1} - e_h^{n-1} )$ and \eqref{eq:three_layer_alt} for the second term, the left-hand side becomes 
\begin{equation*}
  \frac1{2\tau}
  \bigl(
    \| \discreteL^a \partial_\tau e_h^{n+1} \|^2
    - \| \discreteL^a \partial_\tau e_h^{n} \|^2
  \bigr)
  +
  \frac1{2\tau}
  \Gamma_h^{n;\gamma}(u_h^{n+1})
  \bigl( \mathcal{W}_h^{n+1} - \mathcal{W}_h^{n} \bigr) .
\end{equation*}
Writing, for $m \geq 2$,
\begin{multline*}
\Gamma_h^{m;\gamma}(u_h^{m+1}) ( \mathcal{W}_h^{m+1} - \mathcal{W}_h^{m} )
\\
= 
\Gamma_h^{m;\gamma}(u_h^{m+1}) \mathcal{W}_h^{m+1} - \Gamma_h^{m-1;\gamma}(u_h^{m}) \mathcal{W}_h^{m} 
- \bigl( \Gamma_h^{m;\gamma}(u_h^{m+1}) - \Gamma_h^{m-1;\gamma}(u_h^{m}) \bigr) \mathcal{W}_h^{m},
\end{multline*}
so that
the first two terms on the right-hand side telescope, and keeping the term $m=1$ as it stands. Then, similarly we write for $m\geq 2$,
\begin{align*}
  \langle
    \discreteL^{a - \frac12} \delta_{h;\frac34}^m,
    \discreteL^{\frac12 + a} ( e_h^{m+1} -  e_h^{m-1})
  \rangle
  &=
  \langle
    \discreteL^{a - \frac12} \delta_{h;\frac34}^m,
    \discreteL^{\frac12 + a} ( e_h^{m+1} +  e_h^{m})
  \rangle
  -
  \langle
    \discreteL^{a - \frac12} \delta_{h;\frac34}^{m-1},
    \discreteL^{\frac12 + a} ( e_h^{m} +  e_h^{m-1})
  \rangle
  \\
  &\quad
  - \tau
  \langle
    \discreteL^{a - \frac12} \partial_\tau \delta_{h;\frac34}^{m},
    \discreteL^{\frac12 + a} ( e_h^{m} +  e_h^{m-1})
  \rangle,
\end{align*}
such that again the first two terms telescope. 
Inserting the previous three equations into \eqref{eq:error_equation}, summing from $1$ to $n$ and multiplying by $2\tau$ yields
\begin{multline}\label{eq:fully_discrete_summed}
  \| \discreteL^a \partial_\tau e_h^{n+1} \|^2
  +
  \Gamma_h^{n;\gamma}(u_h^{n+1}) \mathcal{W}_h^{n+1}
  =
  \| \discreteL^a \partial_\tau e_h^{1} \|^2
  +
  \Gamma_h^{1;\gamma}(u_h^{2}) \mathcal{W}_h^{1}
  \\
  {}+
  \sum_{m=2}^{n}
  \bigl(
    \Gamma_h^{m;\gamma}(u_h^{m+1}) - \Gamma_h^{m-1;\gamma}(u_h^{m})
  \bigr) \mathcal{W}_h^{m}
  +
  \tau \sum_{m=1}^n
  \langle
    \discreteL^a (\delta_{h;\frac14}^m + S_h^m),
    \discreteL^a (\partial_\tau e_h^{m+1} + \partial_\tau e_h^m)
  \rangle
  \\
  +
  \langle
    \discreteL^{a - \frac12} \delta_{h;\frac34}^n,
    \discreteL^{\frac12 + a} ( e_h^{n+1} +  e_h^{n})
  \rangle
  -
  \langle
    \discreteL^{a - \frac12} \delta_{h;\frac34}^{1},
    \discreteL^{\frac12 + a} ( e_h^{1} +  e_h^{0})
  \rangle
  -\tau
  \sum_{m=2}^n
  \langle
    \discreteL^{a - \frac12} \partial_\tau \delta_{h;\frac34}^{m},
    \discreteL^{\frac12 + a} ( e_h^{m} +  e_h^{m-1})
  \rangle.
\end{multline}
  
  Since $\beta^m \geq \beta_{\min} > 0$ and the weights
$\gamma, 1-2\gamma, \gamma$ are non-negative, the coefficient obeys
$\alpha_{\min} \leq \Gamma_h^{m;\gamma}(u_h^{m+1}) \leq C_\Gamma$, with $C_\Gamma$ from the
proof of \Cref{lem:fully_discrete_local_existence}. Let 
\(
  \mathcal{F}_h^m
  =
  \| \discreteL^{a} \partial_\tau e_h^m \|^2 + \mathcal{W}_h^m
  \textrm{ for }m = 1,\dots,n+1 ,\textrm{ and }
  \mathcal{F}_h^0 \coloneqq \| \discreteL^{\frac12 + a} e_h^0 \|^2 
\)
be the quantities as defined in \Cref{cor:three_layer_energy_control} at $a \in \{0,\tfrac14\}$ and $u_h^m$ replaced by $e_h^m$. 
\begin{subequations}\label{eq:convergence_proof_estimates_summary}
\noindent We now bound the seven terms on the right-hand side of
\eqref{eq:fully_discrete_summed} in turn. By the induction hypothesis
\eqref{eq:fully_discrete_bounds_all_time_steps} at the levels $m = 0,1$, together with
$\Gamma_h^{1;\gamma}(u_h^{2}) \leq C_\Gamma$, \eqref{eq:cons_stiff_end}, Cauchy--Schwarz, and \eqref{eq:three_layer_upper},
\begin{equation}
  \| \discreteL^{a} \partial_\tau e_h^{1} \|^2
  + \Gamma_h^{1;\gamma}(u_h^{2}) \mathcal{W}_h^{1}
  + |\langle \discreteL^{a - \frac12} \delta_{h;\frac34}^1, \discreteL^{\frac12 + a} (e_h^1 + e_h^0)\rangle| 
  \lesssim
  \bigl( h^{p+1-2a} + \tau^2 \bigr)^2 ,
\end{equation}
where we used~$h^{2p} \leq h^{p+1 - 2a}$ for $p \geq 1$ and $h \leq 1$. 
By \eqref{eq:cons_gamma} and \eqref{eq:three_layer_upper},
\begin{equation}
  \sum_{m=2}^{n}
  \bigl| \Gamma_h^{m;\gamma}(u_h^{m+1}) - \Gamma_h^{m-1;\gamma}(u_h^{m}) \bigr| \,
  \mathcal{W}_h^{m}
  \lesssim
  \tau \sum_{m=2}^{n} \mathcal{W}_h^{m}
  \leq
  \tau \sum_{m=1}^{n} \mathcal{F}_h^m .
\end{equation}
By Cauchy--Schwarz, 
\eqref{eq:cons_weak}, the plain Young inequality and $\tau \sum_{m=1}^{n} 1 \leq T$,
\begin{equation}
  \tau \sum_{m=1}^{n}
  \bigl| \langle \discreteL^{a} \delta_{h;\frac14}^m ,
    \discreteL^{a} ( \partial_\tau e_h^{m+1} + \partial_\tau e_h^{m} ) \rangle \bigr|
  \lesssim
  T \bigl( h^{p + 1 - 2a} + \tau^2 \bigr)^2
  +
  \tau \sum_{m=1}^{n+1} \mathcal{F}_h^m .
\end{equation}
By \eqref{eq:cons_stab}, the plain Young
inequality and \eqref{eq:energy_coercive},
\begin{equation}
  \tau \sum_{m=1}^{n}
  \bigl| \langle \discreteL^{a} S_h^m ,
    \discreteL^{a} ( \partial_\tau e_h^{m+1} + \partial_\tau e_h^{m} ) \rangle \bigr|
  \lesssim
  \tau \sum_{m=0}^{n+1} \mathcal{F}_h^m .
\end{equation}
This is the only term whose bound is quadratic in the error with a coefficient that is not
small, and it is therefore the term that forces the use of a Gr\"onwall argument below.
The boundary term involving $\delta_{h;\frac34}^n$ carries no factor $\tau$ and reaches the
level $n+1$, so it is absorbed into the left-hand side. Using \eqref{eq:energy_sum},
$\Gamma_h^{n;\gamma}(u_h^{n+1}) \geq \alpha_{\min}$ and the weighted Young inequality with
weight $\alpha_{\min}/8$ and \eqref{eq:cons_stiff_end}, we have
\begin{align}\label{eq:fully_discrete_boundary_absorption}
  \bigl| \langle \discreteL^{a - \frac12} \delta_{h;\frac34}^{n} ,
    \discreteL^{\frac12 + a} ( e_h^{n+1} + e_h^{n} ) \rangle \bigr|
  &\leq
  \tfrac{2}{\alpha_{\min}} \| \discreteL^{a - \frac12} \delta_{h;\frac34}^{n} \|^2
  +
  \tfrac{\alpha_{\min}}{8} \| \discreteL^{\frac12 + a} ( e_h^{n+1} + e_h^{n} ) \|^2\notag
  \\
  &\lesssim
  \tfrac{2}{\alpha_{\min}} (h^{p + 1 - 2a} + \tau^2)^2
  +
  \tfrac12 \Gamma_h^{n;\gamma}(u_h^{n+1}) \mathcal{W}_h^{n+1} .
\end{align}
Note that this absorption requires no restriction
on $\tau$.
By \eqref{eq:cons_stiff_time}, Young's inequality,
\eqref{eq:energy_coercive} and $\tau \sum_{m=2}^{n} 1 \leq T$,
\begin{equation}
  \tau \sum_{m=2}^{n}
  \bigl| \langle \discreteL^{a - \frac12} \partial_\tau \delta_{h;\frac34}^{m} ,
    \discreteL^{\frac12 + a} ( e_h^{m} + e_h^{m-1} ) \rangle \bigr|
  \lesssim
  T \bigl( h^{p + 1 - 2a} + \tau^2 \bigr)^2
  +
  \tau \sum_{m=0}^{n} \mathcal{F}_h^m .
\end{equation}
\end{subequations}
It is here that the summation by parts pays off. Without it the same term would carry a factor
$\tau^{-1}$ and the defects, being of order $\tau^2$, would not suffice.

Inserting \eqref{eq:convergence_proof_estimates_summary} into
\eqref{eq:fully_discrete_summed}, subtracting the second summand
of~\eqref{eq:fully_discrete_boundary_absorption} from both sides and using
$\Gamma_h^{n;\gamma}(u_h^{n+1}) \geq \alpha_{\min}$, so that the left-hand side
of~\eqref{eq:fully_discrete_summed} is at least
$\min \{ 1, \alpha_{\min} \} \mathcal{F}_h^{n+1}$, we obtain constants $c_1, c_2 > 0$, independent of $n$, $\tau$ and $h$,
with
\begin{equation*}
  \tfrac12 \min \{ 1 , \alpha_{\min} \} \, \mathcal{F}_h^{n+1}
  \leq
  c_1 ( 1 + T ) \bigl( h^{p + 1 - 2a} + \tau^2 \bigr)^2
  +
  c_2 \, \tau \sum_{m=0}^{n+1} \mathcal{F}_h^m .
\end{equation*}
The sum on the right still reaches the level $n+1$. Choosing
$\tau_1 \coloneqq \min \{ 1 , \alpha_{\min} \} / ( 4 c_2 )$ and restricting to
$\tau \leq \tau_1$, that term is absorbed as well, and we arrive at
\begin{equation}\label{eq:fully_discrete_gronwall_ready}
  \mathcal{F}_h^{n+1}
  \leq
  c_3 \bigl( h^{p + 1 - 2a} + \tau^2 \bigr)^2
  +
  c_4 \, \tau \sum_{m=0}^{n} \mathcal{F}_h^m ,
  \qquad
  c_3 , c_4 \text{ independent of } n , \tau , h .
\end{equation}
Every step leading to \eqref{eq:fully_discrete_gronwall_ready} used the hypotheses of the
theorem only at the levels $m \leq n$, together with the bounds
\Cref{lem:fully_discrete_local_existence} supplies at the level $n+1$. The derivation is
therefore valid verbatim with $n$ replaced by any $k \leq n$, so that
\eqref{eq:fully_discrete_gronwall_ready} holds for $\mathcal{F}_h^{k+1}$, $k = 1,\dots,n$, while
$\mathcal{F}_h^0 + \mathcal{F}_h^1 \lesssim ( h^{p + 1 - 2a} + \tau^2 )^2$ by
\eqref{eq:fully_discrete_bounds_all_time_steps} at $m = 0,1$. The discrete Gr\"onwall
inequality applied to this family yields
\begin{equation*}
  \max_{0 \leq m \leq n+1} \mathcal{F}_h^{m}
  \leq
  c_5 \, e^{c_4 T} \bigl( h^{p + 1 - 2a} + \tau^2 \bigr)^2 ,
\end{equation*}
with $c_5$ depending only on the data of the problem. 
Finally, \eqref{eq:energy_coercive} at $m = n+1$ gives
\begin{equation*}
  \| \discreteL^{\frac12 + a} e_h^{n+1} \|^2
  +
  \| \discreteL^{a} \partial_\tau e_h^{n+1} \|^2
  \leq
  ( 1 + \cstcoerc ) \, \mathcal{F}_h^{n+1}
  \lesssim
  \bigl( h^{p + 1 - 2a} + \tau^2 \bigr)^2 ,
\end{equation*}
which is the first asserted estimate, and since $\discreteL$ is uniformly positive definite on
$V_h$, \eqref{eq:discrete_poincare} with $a \in \{0,\tfrac14\}$, $b = a + \frac12$ gives the second,
\begin{equation*}
  \| \discreteL^{a} e_h^{n+1} \|
  \leq
  \lambda_1^{-\frac12} \| \discreteL^{a + \frac12} e_h^{n+1} \|
  \lesssim
  h^{p + 1 - 2a} + \tau^2 . \qedhere
\end{equation*}
\end{proof}

\subsection{Initial data and closure of the induction}
\noindent It remains to specify the discrete initial data and to close the induction. Mirroring
the starting values~\eqref{eq:starting_vec1} of the time-discrete scheme, we set 
\begin{equation}\label{eq:fully_discrete_initial_data}
    u_h^0 = \ritzprojection u_0,
    \qquad
    u_h^1 = \ritzprojection \bigl( u_0 + \tau v_0 + \tfrac{\tau^2}{2} ( f^0 - q^0 \L u_0 ) \bigr)
\end{equation}
that is, the Ritz projection of the Taylor expansion of $u$ about $t = 0$. This is what the
scheme~\eqref{eq:fully_discrete_nonlinear_scheme} at $n = 0$ produces once the ghost value is
eliminated by the centred approximation $u_h^{-1} = u_h^{1} - 2\tau v_0$ of
$\partial_t u ( 0 ) = v_0$ and the stiffness term is evaluated explicitly, so that the level
$n = 1$ is obtained from an explicit step; $q^0 = \alpha^0 + \beta^0 \| \nabla u_0 \|^2$ denotes
the coefficient of the continuous equation at $t = 0$, whence
$f^0 - q^0 \L u_0 = \partial_t^2 u ( 0 )$. 
Since \Cref{thm:fully_discrete_local_convergence} assumes both $\partial_t u$ and
$\partial_t^2 u$ in $L^\infty ( 0,T;D ( \L^{\frac34} ) )$, we have
$\partial_t u \in W^{1,\infty} ( 0,T;D ( \L^{\frac34} ) ) \hookrightarrow C ( [0,T];D ( \L^{\frac34} ) )$,
so that $\partial_t u$ has a $D ( \L^{\frac34} )$-valued trace at $t=0$ and
$v_0 = \partial_t u(0) \in D ( \L^{\frac34} )$. This is therefore not an additional hypothesis
below.

\begin{lemma}[Error bounds for the fully discrete initial data]\label[lemma]{lem:fully_discrete_initial_conditions}
  Let \Cref{ass:existence_exact_solution} hold, let $a\in\{0,\frac14\}$, and let
  $u_h^0, u_h^1$ be given by~\eqref{eq:fully_discrete_initial_data}. Then, for $m = 0,1$
  \begin{align*}
    \| \discreteL^{\frac12 + a} \bigl( \ritzprojection u(t_m) - u_h^m \bigr) \|
    +
    \| \discreteL^{a} \bigl( \ritzprojection u(t_m) - u_h^m \bigr) \|
    &{}\lesssim h^{p + 1 - 2a} + \tau^3\,,
    \\
    \| \discreteL^{a} \partial_\tau \bigl( \ritzprojection u(t_1) - u_h^1 \bigr) \|
    &{}\lesssim h^{p + 1 - 2a} + \tau^2\,,
  \end{align*}
  and the uniform bounds \eqref{eq:fully_discrete_bounds_all_time_steps} hold for
  $m = 0,1$, for $\tau$ and $h$ sufficiently small. 
\end{lemma}

\begin{proof}  
  Write $e_h^m = \ritzprojection u(t_m) - u_h^m$. At $m=0$ we have
  $u_h^0 = \ritzprojection u_0 = \ritzprojection u(t_0)$, so $e_h^0 = 0$ and the first estimate
  is trivial. At $m=1$, since $\ritzprojection$ is linear,
  \begin{equation*}
    e_h^1
    =
    \ritzprojection ( u(t_1) - u_0 - \tau v_0 - \tfrac{\tau^2}{2}\bigl( f^0 - q^0\, \L u_0 \bigr) )\,,
  \end{equation*}
  and the bracket is estimated by Taylor expansion about $t=0$, which, together with
  $\partial_t^2 u(0) = f^0 - q^0\L u_0$, is bounded by
  $\tau^3 \| \L^{b} \partial_t^3 u \|_{L^\infty(L^2)}$ in the
  $\L^{b}$-norm for $b \in \{ \frac12, \tfrac12 + a \}$. Applying
  \eqref{eq:fully_discrete_fractional_stability_ritz_below_half} at $a$ and
  \eqref{eq:fully_discrete_fractional_stability_ritz} at $\tfrac12 + a$ to pass from
  $\ritzprojection$ to the continuous norm gives
  \begin{equation*}
    \| \discreteL^{a} e_h^1 \| \lesssim \tau^3 \| \L^{\frac12} \partial_t^3 u \|_{L^\infty(L^2)}\,,
    \qquad
    \| \discreteL^{\frac12 + a} e_h^1 \| \lesssim \tau^3 \| \L^{\frac12 + a} \partial_t^3 u \|_{L^\infty(L^2)}\,,
  \end{equation*}
  \noindent both $\lesssim \tau^3$ by the standing hypothesis
  $\partial_t^3 u \in L^\infty ( 0,T; D ( \L^{\frac34} ) )$, which also controls the
  $\L^{\frac12}$-norm because $\| \L^{\frac12} w \| \leq
  \lambda_1^{-\frac14} \| \L^{\frac34} w \|$. %
  Since $e_h^0 = 0$, we have $\partial_\tau e_h^1 = \tau^{-1} e_h^1$, whence
  $\| \discreteL^{a} \partial_\tau e_h^1 \| \lesssim \tau^2$. %
  Note that from the proof, for this choice of starting values the
  error relative to $\ritzprojection u$ carries no spatial contribution at all. The asserted bounds are purely temporal ones,
  written in the form in which \Cref{cor:fully_discrete_convergence,cor:fully_discrete_optimal_error} consume them. 

  The uniform bounds
  \eqref{eq:fully_discrete_bounds_all_time_steps} follow by the triangle inequality from these
  estimates, the stability of $\ritzprojection$ (\eqref{eq:fully_discrete_fractional_stability_ritz_below_half} for $a$ and
  \eqref{eq:fully_discrete_fractional_stability_ritz} for $\tfrac12+a$) and the factor $2$ in
  the definition of $M_a$ and $M'_a$. 
\end{proof}

\begin{corollary}[Convergence of the fully discrete scheme]\label[corollary]{cor:fully_discrete_convergence}
  Let the assumptions of \Cref{thm:fully_discrete_local_convergence} hold, in
  particular, the mesh condition $\tau \leq \kappa h$ at the endpoint case $\theta = \tfrac14$, and let the
  discrete initial data be given by \eqref{eq:fully_discrete_initial_data}. Then there exist
  $\tau_0, h_0 > 0$, independent of $n$, such that for all $0 < \tau \leq \tau_0$ and
  $0 < h \leq h_0$ the fully discrete solution $\{ u_h^{n} \}_{n=0}^{N}$ exists, is unique, and
  satisfies
  \begin{equation*}
    \max_{0 \leq n \leq N}
      \| \discreteL^{\frac12 + a} \bigl( \ritzprojection u(t_n) - u_h^{n} \bigr) \|
      +
      \max_{1 \leq n \leq N}
      \| \discreteL^{a} \partial_\tau \bigl( \ritzprojection u(t_n) - u_h^{n} \bigr) \|
    \lesssim
    h^{p+1 - 2a} + \tau^2\,.
  \end{equation*}
\end{corollary}

\begin{proof}
  Let $n^\ast$ be the largest index for which the hypotheses of
  \Cref{thm:fully_discrete_local_convergence}, \ie{}, the uniform bounds
  \eqref{eq:fully_discrete_bounds_all_time_steps} and the error estimate at all levels
  $m \leq n^\ast$ are satisfied. By \Cref{lem:fully_discrete_initial_conditions} they hold
  for $m = 0,1$, so $n^\ast \geq 1$. If $n^\ast < N$, then
  \Cref{lem:fully_discrete_local_existence} provides $u_h^{n^\ast+1}$ and
  \Cref{thm:fully_discrete_local_convergence} extends the error estimate to that level, the
  uniform bounds following from it by the triangle inequality together with the stability of
  $\ritzprojection$ and the factor $2$ in the definition of $M_a$ and $M'_a$, contradicting
  the maximality of $n^\ast$. Hence $n^\ast = N$. The argument is admissible because the constants in
  \Cref{thm:fully_discrete_local_convergence} are independent of $n$, $\tau$ and $h$, so a
  single pair $(\tau_0, h_0)$ works at every step.
\end{proof}

\begin{corollary}[Optimal-order error of the fully discrete scheme]\label[corollary]{cor:fully_discrete_optimal_error}
  Under the assumptions of \Cref{cor:fully_discrete_convergence} the full error of
  the fully discrete scheme satisfies
  \begin{equation*}
    \max_{0 \leq n \leq N} \| u(t_n) - u_h^{n} \|
    \lesssim
    h^{p+1} + \tau^2 ,
    \qquad
    \max_{0 \leq n \leq N} \| \nabla ( u(t_n) - u_h^{n} ) \|
    \lesssim
    h^{p} + \tau^2 ,
  \end{equation*}
  \noindent both of which are of optimal order in $h$ for the space $V_h$ of continuous
  piecewise polynomials of degree~$p$.
\end{corollary}

\begin{proof}
  We split $u(t_n) - u_h^{n} = ( \hat u^{n} - \ritzprojection \hat u^{n} ) + e_h^{n}$. The
  approximation estimates for the Ritz projection with $s=p+1$, admissible because
  $u \in L^\infty ( 0,T;H^{p+1}(\domain) )$ and $u(t) \in H_0^1(\domain)$, give
  \begin{equation*}
    \| \hat u^{n} - \ritzprojection \hat u^{n} \|
    +
    h \| \nabla ( \hat u^{n} - \ritzprojection \hat u^{n} ) \|
    \lesssim
    h^{p+1} | \hat u^{n} |_{H^{p+1}(\domain)}
    \lesssim
    h^{p+1} .
  \end{equation*}
  \noindent For the remaining part, $e_h^{n} \in V_h$, so that
  \eqref{eq:discrete_poincare} with $a=0$ and $b=\tfrac12$ together with
  \Cref{cor:fully_discrete_convergence} yields
  $\| e_h^{n} \| \leq \lambda_1^{-\frac12} \| \discreteL^{\frac12} e_h^{n} \|
  \lesssim h^{p+1} + \tau^2$, while
  $\| \nabla e_h^{n} \| = \| \discreteL^{\frac12} e_h^{n} \| \lesssim h^{p+1} + \tau^2$. The
  triangle inequality now gives both assertions, the second because
  $h^{p+1} \leq h^{p}$ for $h \leq 1$.
\end{proof}

\section{Numerical experiments}\label{sec:numerics}

\noindent In this section we test the fully discrete scheme numerically on manufactured
solutions. Throughout, $\domain = \left( 0,1 \right)^2$, the
spatial discretization uses continuous Lagrange elements of degree $p=1,2,3$, and the source term
$f$ is constructed from the prescribed exact solution and the coefficients given in the respective examples. The parameter $\lambda \in \mathbb{N}$ controls the spatial oscillation of the exact solution. 

The scheme~\eqref{eq:fully_discrete_nonlinear_scheme} is used with
$\theta = \gamma = \tfrac12$ and $\eta = \frac14$, so that it reduces to the symmetric three-layer
form~\eqref{eq:nonlinear_scheme_extended}, in particular
$\theta > \tfrac14$, so no mesh condition is active. The spatial convergence studies use the
uniform mesh sequence $h = 2^{-2}, 2^{-3}, \dots, 2^{-6}$, and use a fixed number of time steps~$N=2^{12}$ such that $\tau=T2^{-12}$ is defined accordingly. This keeps the temporal
contribution below the spatial error for most data points. The
temporal convergence studies use the fixed mesh $h = 2^{-6}$ and $p = 3$ for time steps
$\tau = T 2^{-k}$, $k = 9, \dots , 12$. 

The nonlinear system at each time level is solved by the Picard
iteration~\eqref{eq:fully_discrete_nonlinear_iteration_scheme}, and the linear system of each
Picard step by a conjugate-gradient (cg-) method with relative and absolute tolerance
$10^{-12}$ and at most $1000$ iterations. The Picard loop is started from the analysed initial
guess $u_{h;0}^{n+1} = u_h^{n}$ and is stopped as soon as
\begin{equation}\label{eq:picard_stopping}
  \max \left\{
    \frac{| \Gamma_h^{n;\gamma}(u_{h;\ell+1}^{n+1}) - \Gamma_h^{n;\gamma}(u_{h;\ell}^{n+1}) |}
         {\max \bigl\{ | \Gamma_h^{n;\gamma}(u_{h;\ell+1}^{n+1}) |, 1 \bigr\}} ,\;
    \frac{\| u_{h;\ell+1}^{n+1} - u_{h;\ell}^{n+1} \|_\infty}
         {\max \bigl\{ \| u_{h;\ell}^{n+1} \|_\infty, 1 \bigr\}}
  \right\}
  < \varepsilon_{\mathrm{NL}} \coloneqq 10^{-10} ,
\end{equation}
with at most $L_{\max} = 10$ iterations permitted. Note that here $\| \cdot \|_\infty$ denotes the maximum over the coefficient vectors with respect to the nodal basis of~$V_h$. 
In every run reported below the criterion was met before $L_{\max}$ was
reached, at every time level. Note that the tolerances for the nonlinearity and the cg-solve are tighter than required such that the errors are not polluted in any form. 

The discrete initial data follow the Taylor
start~\eqref{eq:fully_discrete_initial_data}.
The numerical experiments are implemented in Python using the FEniCSx software framework~(\cite{Baratta2023DOLFINx,Scroggs2022DOFMaps,Scroggs2022Basix,Alnaes2014UFL}). The versions of all software dependencies are specified in the enclosed~\texttt{pyproject.toml} file in the repository. The scheme is realised in the auxiliary-variable form of~\Cref{rem:implementation}. %
Errors are evaluated at the final time only, which reduces the cost per run. This is
consistent with the theory, whose estimates are uniform over the time levels and therefore
hold at the final time in particular. Further, it is the natural choice in practice, since the
accumulation of the error means that the largest relative error is typically attained
there.

\begin{remark}[Implementation]\label{rem:implementation}
  In practice, we try to avoid reassembling the right-hand side in every Picard iteration of~\eqref{eq:fully_discrete_nonlinear_iteration_scheme} such that the assembly cost is only paid at each time step~$n$. This can be achieved by introducing the auxiliary variable %
  \begin{equation}\label{eq:fully_discrete_auxiliary_solution}
    v_{h}^{n+1}
    \coloneqq
    u_h^{n+1;\theta}.
  \end{equation}
  The idea is that this auxiliary variable cancels the terms that contain the nonlinearity on the right-hand side of~\eqref{eq:fully_discrete_nonlinear_iteration_scheme}. 
  Multiplying~\eqref{eq:fully_discrete_nonlinear_scheme} by~$\tau^2\theta$ and plugging~\eqref{eq:fully_discrete_auxiliary_solution} into~\eqref{eq:fully_discrete_nonlinear_scheme}, the nonlinear scheme reads 
  \begin{equation*}
    v_h^{n+1}
    +
    \tau^2 \theta \Gamma_h^{n;\gamma}(u_{h}^{n+1}) \discreteL v_h^{n+1}
    =
    \tau^2 \theta f^{n+1;\eta} + u_h^n.
  \end{equation*}
  The weighted source term~$f^{n+1;\eta} = \eta f((n+1)\tau) + (1-2\eta)f(n\tau) + \eta f((n-1)\tau)$ is realised by the standard quadrature used in FEniCSx of the $L^2$ scalar product of the right-hand side. This introduces only a quadrature error which is of greater order than an interpolation error and can thus be disregarded. %
  Then, we solve this nonlinear equation for every time step with the following Picard iteration: We seek a sequence of functions~$\{v_{h;\ell +1}^{n+1}\}_{\ell \geq 0} \subset V_h$ that solves the \emph{fully discrete linear iteration scheme} 
  \begin{equation*}
    v_{h;\ell +1}^{n+1}
    +
    \tau^2 \theta 
    \Gamma_h^{n;\gamma}(u_{h;\ell }^{n+1})
    \discreteL v_{h;\ell +1}^{n+1}
    =
    \tau^2 \theta f^{n+1;\eta} + u_h^n
  \end{equation*}
  with the initial guess~$u_{h;0}^{n+1} = u_h^n$ and update step %
  \(
    \theta u_{h;\ell +1}^{n+1}
    =
    v_{h;\ell +1}^{n+1}
    - (1 - 2\theta) u_{h}^n
    - \theta u_h^{n-1}.
  \)
\end{remark}

\begin{figure}[!ht]
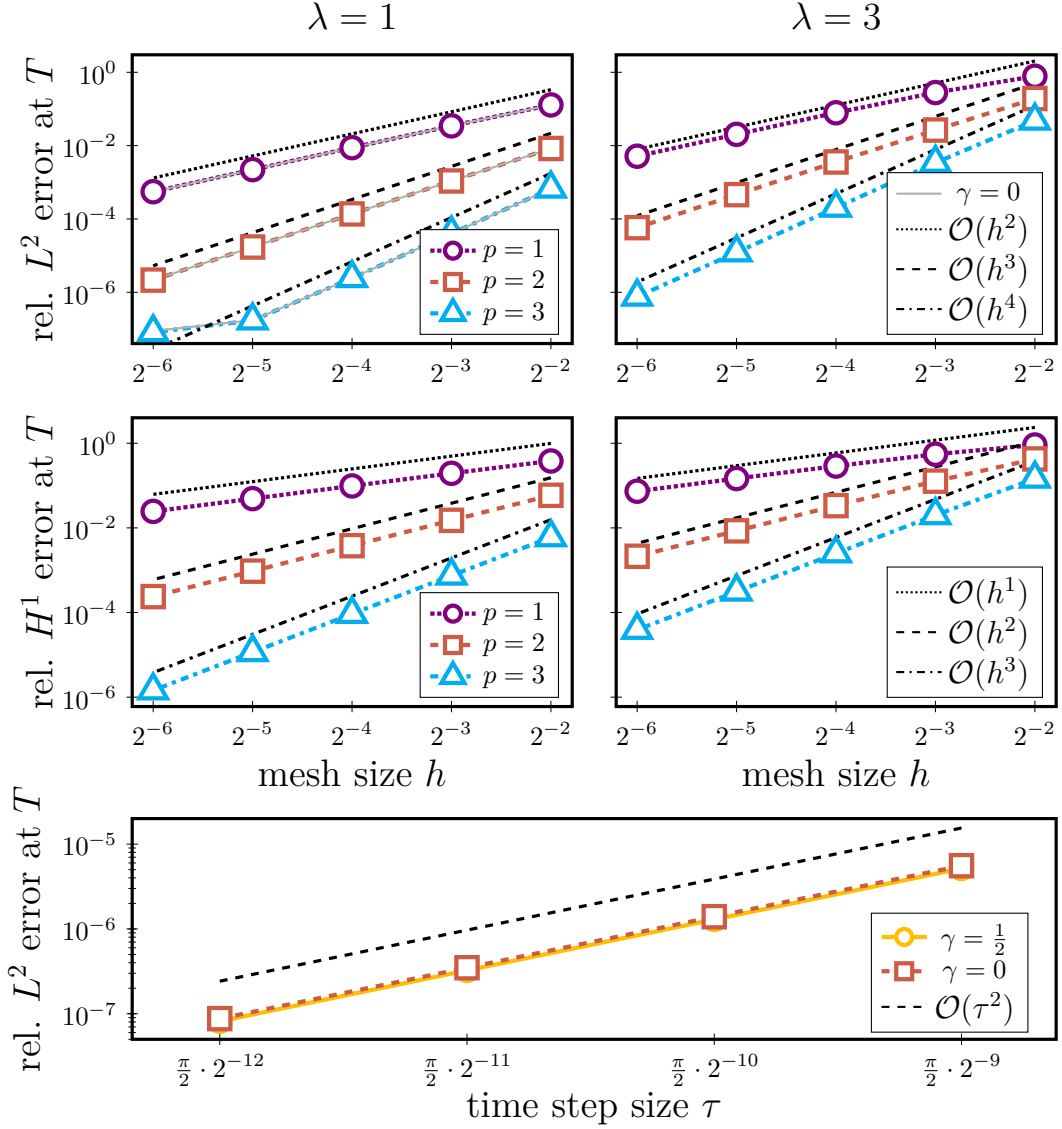

	\centering
	\includestandalone{example_4}
    \caption{Example~1: relative $L^2$ errors (top and bottom) and $H^1_0$ errors (mid) at the last time step $T=\frac\pi{2}$ for $\lambda=1$ (left) and $\lambda=3$ (right). The order $p+1$ for $L^2$ and $p$ for $H^1$ in space are attained, with the exception of the last $p=3$ point for $\lambda=1$, where the fixed time step contributes at the same order. %
    Further, the bottom plot shows the second order convergence in time predicted by the theory. %
    Additionally, the plot on the top left and the plot on the bottom show the same convergence behaviour and similarly sized error for the locally linear scheme~$\gamma=0$.}\label{fg:example_4}
\end{figure}

\paragraph{Example~1: Optimal spatial and temporal convergence studies.} We present a first very simple constructed example to show that the proposed method converges with the optimal orders. We set the time interval $[0,\tfrac\pi{2}]$ with $\alpha\left( t \right)=1+t^2$, and $\beta\left( t \right)=2-\sin(t)$. Then we solve the PDE with exact solution 
\begin{equation*}
	u(x, y, t) = \sin(\lambda \pi x)\sin(\lambda \pi y)\sin(t).
\end{equation*}

\noindent \Cref{fg:example_4} reports the relative $L^2$ (top and bottom) and $H^1_0$ (middle)
errors at the final time~$T=\frac\pi2$ for~$\lambda=1$ (left) and~$\lambda=3$ (right), for
polynomial degrees $p=1,2,3$ under uniform mesh refinement. It confirms the optimal spatial
orders~$p+1$ in $L^2$ and~$p$ in $H^1$ predicted by \Cref{cor:fully_discrete_optimal_error}.
The time step size is fixed at~$\tau=\frac\pi{2}\cdot2^{-12}$. Note that the final point for
$h=2^{-6}$ and $p=3$ does not attain the optimal order, because the error has reached the
temporal plateau. The bottom plot shows that at this time step size the temporal contribution
is of the same order of magnitude. That final point is therefore not a pure spatial error, and
the degradation of the rate is the expected consequence rather than a property of the method.
The bottom plot also shows the behaviour under refinement of the time step at a fixed mesh,
for~$p=3$ and~$h=2^{-6}$. The error decreases at the second order predicted by
\Cref{cor:fully_discrete_optimal_error} until it reaches the level of the spatial
discretization error, after which further refinement in time brings no gain and the observed
rate deteriorates.
The observed orders in space and time are independent of the spatial oscillation
parameter~$\lambda$, although a larger~$\lambda$ increases the size of the error
in~\Cref{fg:example_4}. Increasing~$\lambda$ has no significant effect on the Picard iteration
either: the stopping criterion~\eqref{eq:picard_stopping} was met at every time step,
and~$L_{\max}$ was never reached.

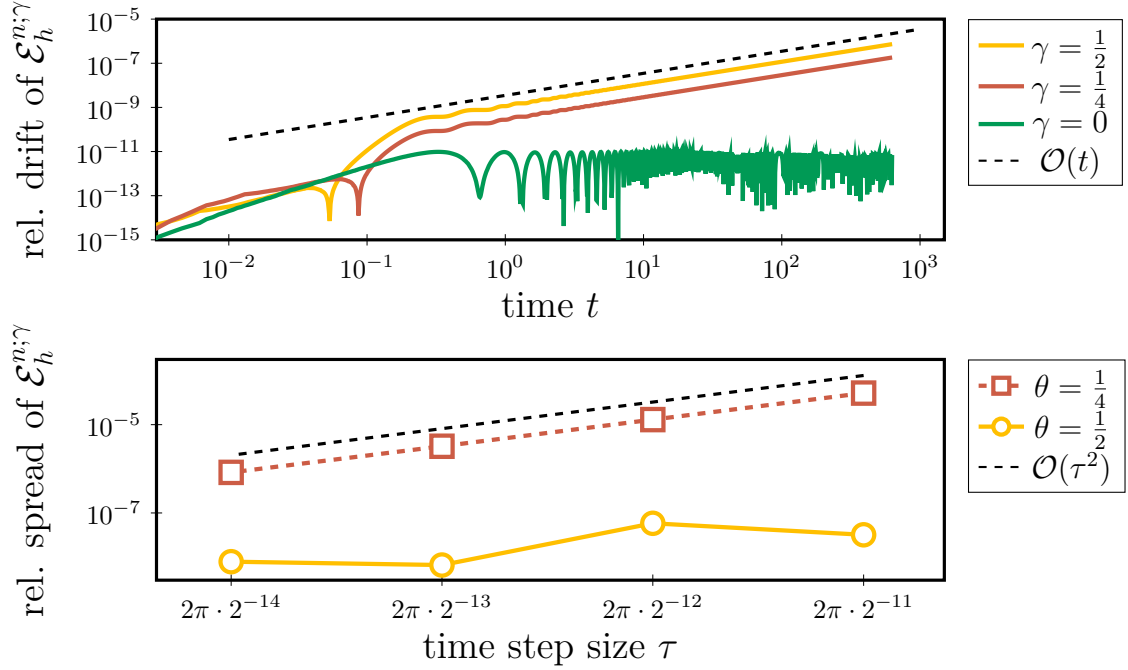
\begin{figure}[!ht]
	\centering
	\begin{tikzpicture}

\definecolor{darkgray176}{RGB}{176,176,176}
\definecolor{darkcoral}{rgb}{0.8, 0.36, 0.27}
\definecolor{amber}{rgb}{1.0, 0.75, 0.0}

\begin{scope}[shift={(0cm,0cm)}]
\begin{axis}[
width=12cm,
height=4.5cm,
axis line style=very thick,
log basis x={10},
log basis y={10},
tick align=center,
tick pos=left,
legend columns=1,
legend style={/tikz/column 2/.style={column sep=7pt,}, font=\large},
legend pos=outer north east,
xlabel={{\Large time $t$}},
x label style={at={(axis description cs:0.5,-0.2)},anchor=north},
x grid style={darkgray176},
xmin=3e-3, xmax=1.5e3,
xmode=log,
xtick style={color=black},
xtick={1e-2,1e-1,1e0,1e1,1e2,1e3},
xticklabels={
  \(\displaystyle {10^{-2}}\),
  \(\displaystyle {10^{-1}}\),
  \(\displaystyle {10^{0}}\),
  \(\displaystyle {10^{1}}\),
  \(\displaystyle {10^{2}}\),
  \(\displaystyle {10^{3}}\),
},
ylabel={{\Large rel. drift of $\mathcal{E}_h^{n;\gamma}$}},
y label style={at={(axis description cs:-0.12,0.5)},anchor=south},
y grid style={darkgray176},
ymin=1e-15, ymax=1e-5,
ymode=log,
ytick style={color=black},
ytick={1e-15,1e-13,1e-11,1e-9,1e-7,1e-5},
yticklabels={
  \(\displaystyle {10^{-15}}\),
  \(\displaystyle {10^{-13}}\),
  \(\displaystyle {10^{-11}}\),
  \(\displaystyle {10^{-9}}\),
  \(\displaystyle {10^{-7}}\),
  \(\displaystyle {10^{-5}}\),
}
]

\addlegendimage{amber,ultra thick,solid}\addlegendentry{$\gamma=\tfrac12$}
\addlegendimage{darkcoral,ultra thick,solid}\addlegendentry{$\gamma=\tfrac14$}
\addlegendimage{ForestGreen,ultra thick,solid}\addlegendentry{$\gamma=0$}
\addlegendimage{no markers,black,thick,dashed}\addlegendentry{$\mathcal{O}(t)$}

\addplot [amber, ultra thick, solid] table[x=t,y=drift,col sep=comma]
  {./figures/energy_drift_gamma05.csv};

\addplot [darkcoral, ultra thick, solid] table[x=t,y=drift,col sep=comma]
  {./figures/energy_drift_gamma025.csv};

\addplot [ForestGreen, ultra thick, solid] table[x=t,y=drift,col sep=comma]
  {./figures/energy_drift_gamma0.csv};

\addplot [very thick, black, dashed] coordinates {(1e-2, 3.5e-11) (1e3, 3.5e-6)};

\end{axis}
\end{scope}

\begin{scope}[shift={(0cm,-4.5cm)}]
\begin{axis}[
width=12cm,
height=4.5cm,
axis line style=very thick,
log basis x={10},
log basis y={10},
tick align=center,
tick pos=left,
legend columns=1,
legend style={/tikz/column 2/.style={column sep=7pt,}, font=\large},
legend pos=outer north east,
xlabel={{\Large time step size $\tau$}},
x label style={at={(axis description cs:0.5,-0.2)},anchor=north},
x grid style={darkgray176},
xmin=3.0e-4, xmax=4.0e-3,
xmode=log,
xtick style={color=black},
xtick={0.003067961572265625,0.0015339807861328126,0.0007669903930664063,0.00038349519653320315},
xticklabels={
  \(\displaystyle {2\pi\cdot2^{-11}}\),
  \(\displaystyle {2\pi\cdot2^{-12}}\),
  \(\displaystyle {2\pi\cdot2^{-13}}\),
  \(\displaystyle {2\pi\cdot2^{-14}}\),
},
ylabel={{\Large rel. spread of $\mathcal{E}_h^{n;\gamma}$}},
y label style={at={(axis description cs:-0.12,0.5)},anchor=south},
y grid style={darkgray176},
ymin=3e-9, ymax=3e-4,
ymode=log,
ytick style={color=black},
ytick={1e-9,1e-7,1e-5,1e-3},
yticklabels={
  \(\displaystyle {10^{-9}}\),
  \(\displaystyle {10^{-7}}\),
  \(\displaystyle {10^{-5}}\),
  \(\displaystyle {10^{-3}}\),
}
]

\addlegendimage{darkcoral,ultra thick,dashed,every mark/.append style={fill=white,solid},mark=square*,mark size=3}\addlegendentry{$\theta=\tfrac14$}
\addlegendimage{amber,ultra thick,solid,every mark/.append style={fill=white,solid},mark=*,mark size=3}\addlegendentry{$\theta=\tfrac12$}
\addlegendimage{no markers,black,thick,dashed}\addlegendentry{$\mathcal{O}(\tau^2)$}

\addplot [darkcoral,ultra thick,dashed,every mark/.append style={fill=white,solid},mark=square*,mark size=4]
  table[x=tau,y=spread,col sep=comma]{./figures/energy_spread_theta025.csv};

\addplot [amber,ultra thick,solid,every mark/.append style={fill=white,solid},mark=*,mark size=4]
  table[x=tau,y=spread,col sep=comma]{./figures/energy_spread_theta05.csv};

\addplot [very thick, black, dashed]
table {
  0.003067961572265625 0.00013
  0.00038349519653320315 2.03125e-06
};

\end{axis}
\end{scope}

\end{tikzpicture}
    \caption{Example~2: relative drift (top) and relative spread (bottom) of the fully discrete
energy. We observe that for $\theta=\frac{1}{2}$ the energy is conserved over $100$ periods up
to an error induced by the Picard tolerance, which grows linearly in time. Further, we observe
for $\gamma=\frac12$ that the spread of the energy does not depend on the time step
size~$\tau$ when $\theta=\frac12$, whereas for $\theta=\frac14$ the energy is conserved only to
second order in~$\tau$.}\label{fg:example_5_energy}
\end{figure}

\paragraph{Example~2: Conservation of the discrete energy.} In this example we observe the conservation of
the fully discrete energy. We choose constant coefficients $\alpha=\beta=1$ and a vanishing
source $f$. Further, the initial conditions are given by $u_0=0$ and
$v_0=\sin(\pi x)\sin(\pi y)$. %
We choose $h=2^{-5}$ and $p=2$, with $T=10^2\cdot 2\pi$ and $\tau = 2\pi\cdot 2^{-13}$ for the
conservation study, and $\tau=2\pi\cdot 2^{-k}$ for $k=11,\dots,14$ for the convergence
study. %

In this setting the scheme conserves the fully discrete energy
of~\Cref{prop:fully_discrete_conservation} for the parameter~$\theta=\frac12$. The top plot
of~\Cref{fg:example_5_energy} shows the relative drift of the energy, that is, the relative
difference between the energy at time~$t$ and the energy at the initial time. %
We observe that over $100$ periods the energy is conserved for the three parameters
$\gamma\in\{0,\frac14,\frac12\}$. Since the identity holds for the exact solution of the
nonlinear scheme only, the remaining error grows linearly in~$t$. It stems from the nonlinear
solve and is governed by the Picard tolerance, set to
$\varepsilon_{\textrm{NL}}=10^{-12}$ with $L_{\max} = 30$. %
The case $\gamma=0$ confirms this reading. By~\Cref{rem:gamma-half} the scheme is then
linearly implicit, no Picard iteration is performed, and the energy is conserved all the same.
Accordingly, its drift stays at the level of round-off over the whole interval and shows no
growth in~$t$, whereas the other two curves grow linearly at the rate set
by~$\varepsilon_{\textrm{NL}}$. %

The bottom plot of~\Cref{fg:example_5_energy} shows the relative spread of the energy over the
whole run, that is, the largest relative difference between any two of its values, plotted
against the time step size~$\tau$. We observe that for $\gamma=\frac12$ this spread is
independent of~$\tau$ when $\theta=\frac12$, whereas for $\theta\neq\frac12$ the energy is not
conserved but approximated to second order in~$\tau$, as predicted
by~\Cref{rem:conservation_scope}.

\section{Conclusion}\label{sec:conclusion}
\noindent We have combined the symmetric three-layer time discretization of
\cite{KirchhoffTime} with a conforming finite element method of arbitrary polynomial degree
$p$ and a Picard iteration for the nonlinearity, and analysed the resulting fully discrete
scheme for a two-dimensional Kirchhoff-type equation with time-dependent coefficients. The
nonlinear equation solved at each time level was shown to have a unique solution, to which the
Picard iteration converges (\Cref{lem:fully_discrete_local_existence}). For the homogeneous problem with constant
coefficients and the parameter choice $\theta = \tfrac12$, the
exact solution of the scheme conserves the discrete energy $\mathcal{E}_h^{n;\gamma}$ identically
(\Cref{prop:fully_discrete_conservation}). This is the fully discrete counterpart of the
temporal invariants of~\cite{KirchhoffTime}, see \Cref{rem:gamma-half}, and it survives the spatial discretization because
the method is conforming. 
We further proved convergence of optimal order $\bigO ( h^{p+1} + \tau^2 )$ in $L^2$ and $\bigO ( h^{p} + \tau^2 )$ in $H^1$ (\Cref{cor:fully_discrete_optimal_error}) for the \emph{full} error $u(t_n) - u_h^n$. 
The numerical experiments of \Cref{sec:numerics} exhibit
exactly these rates: second order in time up to the point at which the spatial error becomes
dominant, and the full order $h^{p+1}$ in the $L^2$ norm of the full error. 

Several questions remain open. Both convergence results hold at every level of the
grid on the interval $[0,T]$ of \Cref{ass:existence_exact_solution}, itself local since
$T < T^{\star\star}$. Extending them up to the blow-up time would require a different argument.
Both also rest on regularity going well beyond that assumption. 
\Cref{thm:fully_discrete_local_convergence} demands four time derivatives in the
operator-domain scale and $H^{p+1}$ spatial regularity of $u$, $\partial_t u$,
$\partial_t^2 u$ and $f$, and whether the optimal rate survives under weaker hypotheses we do
not know. At the endpoint $\theta = \tfrac14$ the three-layer displacement energy is non-coercive
in the displacement norm, whether it is necessary for
the \emph{scheme} or only for this method of proof is open. At
$\theta  = \tfrac12$ the scheme conserves $\mathcal{E}_h^{n;\gamma}$ exactly
(\Cref{prop:fully_discrete_conservation}), but we have established no energy inequality for the
whole range with time-dependent coefficients, and claim no unconditional stability there.

Verifying \Cref{prop:fully_discrete_conservation} for the finite
element spaces used in the convergence studies, where the Picard and conjugate-gradient
tolerances also enter, lies beyond the present scope, as does a systematic study of the
sensitivity to those tolerances.

\section*{Declarations}
\ifblind
\else
\noindent\textbf{Funding.}
The work of B.~D\"orich was funded by the Deutsche Forschungsgemeinschaft (DFG, German Research Foundation) -- Project-ID 258734477 -- SFB 1173.
 This work by A.~Rupp was supported by the Deutsche Forschungsgemeinschaft
(DFG, German Research Foundation) -- Project-ID 577175348. And in part supported by the Deutsche
Forschungsgemeinschaft under Germany's Excellence Strategy -- EXC-2047/2 -- 390685813, through
F.~Krumbiegel's and A.~Rupp's stay at the Hausdorff Research Institute for Mathematics. {Additionally, Z.~Vashakidze was supported by the Shota Rustaveli National Science Foundation of Georgia (SRNSFG) under Grant No. FR-25-215.} The authors declare no further funding for this work.
\fi

\noindent\textbf{Competing interests.} The authors declare that they have no competing
interests.

\noindent\textbf{Data and code availability.} \paperdataavailability

\noindent\textbf{Declaration of generative AI and AI-assisted technologies.} \paperdeclarationAI

%

		
		
		\phantomsection                
		
		\printbibliography[
		heading=bibliography,   
		title={References}      
		]                         
		
		\newif\ifForceBibBookmark
		\ForceBibBookmarktrue     
		
		\ifForceBibBookmark
		\phantomsection                  
		\pdfbookmark[1]{References}{bib} 
		\fi

\appendix
\section*{Appendix}\label{sec:appendix}
\setcounter{section}{1}

\subsection{Proof of the full-order convergence} \label{app:defects}

\begin{proof}[Proof of \Cref{lem:consistency}]
  In a first step, we estimate the difference in the nonlinearity in the first sum on the
  right-hand side of \eqref{eq:fully_discrete_summed}. Abbreviating
  $\Lambda^m \coloneqq \gamma \| \discreteL^\frac12 u_h^{m+1} \|^2
  + (1 - 2\gamma) \| \discreteL^\frac12 u_h^{m} \|^2
  + \gamma \| \discreteL^\frac12 u_h^{m-1} \|^2$, so that
  $\Gamma_h^{m;\gamma}(u_h^{m+1}) = \alpha^m + \beta^m \Lambda^m$. Adding and subtracting
  $\beta^m \Lambda^{m-1}$, we have that
  \begin{subequations}\label{eq:cons_gamma_estimate}
  \begin{equation}\label{eq:cons_gamma_diff}
    \bigl|
      \Gamma_h^{m;\gamma}(u_h^{m+1}) - \Gamma_h^{m-1;\gamma}(u_h^{m})
    \bigr|
    \leq
    \bigl| \alpha^m - \alpha^{m-1} \bigr|
    +
    \bigl| \beta^m - \beta^{m-1} \bigr| \, \Lambda^{m-1}
    +
    \beta_{\max} \bigl| \Lambda^{m} - \Lambda^{m-1} \bigr| .
  \end{equation}
  The first two terms are $\lesssim \tau$ by the Taylor expansion of the coefficients
  $\alpha, \beta \in C^1[0,T]$ and the boundedness of the fully discrete solution at the
  previous time steps, using
  \(\| \discreteL^\frac12 \varphi_h \|^2 = \langle\discreteL^\frac34\varphi_h,
  \discreteL^\frac14 \varphi_h\rangle \). For the third we apply the proportionality argument \eqref{eq:fully_discrete_nonlinearity_proportional} 
  to each of the three pairs of consecutive levels making up $\Lambda^{m} - \Lambda^{m-1}$,
  \begin{align}
    \| \discreteL^\frac12 u_h^{m+1} \|^2
    - \| \discreteL^\frac12 u_h^{m} \|^2
    &=
    \langle
      \discreteL^\frac34 (u_h^{m+1} + u_h^m),
      \discreteL^\frac14 (u_h^{m+1} - u_h^{m})
    \rangle \notag
    \\
    &\leq
    \tau
    \bigl(
      \| \discreteL^\frac34 u_h^{m+1} \|
      +\| \discreteL^\frac34 u_h^{m} \|
    \bigr)
    \| \discreteL^\frac14 \partial_\tau u_h^{m+1} \|
    \lesssim \tau ,
  \end{align}
\end{subequations}
  and likewise at the levels $m,m-1$ and $m-1,m-2$, so that
  $| \Gamma_h^{m;\gamma}(u_h^{m+1}) - \Gamma_h^{m-1;\gamma}(u_h^{m}) | \lesssim \tau$. For $m < n$ these norms come
  from~\eqref{eq:fully_discrete_bounds_all_time_steps}, whereas at the top index $m = n$, where
  $\Lambda^{n}$ reaches $u_h^{n+1}$ and $\partial_\tau u_h^{n+1}$, they come from the generic
  bounds \Cref{lem:fully_discrete_local_existence} supplies for the level it constructs,
  uniformly in $n$, $h$ and $\tau$. Overall, \eqref{eq:cons_gamma_estimate} gives the bound~\eqref{eq:cons_gamma}. Note that this bound does not depend on the parameter~$a$. 
  
  Next, we bound the defects~$\delta_{h;\frac14}^m$. The defect consists of the terms 
  \begin{subequations}\label{eq:cons_weak_estimate}
  \(
    \delta_{h;\frac14}^m
    =
    \delta_{h;f}^m
    +
    \delta_{h;1}^m.
  \)
  Then, for the first defect~$\delta_{h;f}^m$ we have that 
  \begin{multline}
    \langle
      \delta_{h;f}^m, \discreteL^{2a} (\partial_\tau e_h^{m+1} + \partial_\tau e_h^m)
    \rangle
    =
    \langle
      \discreteL^a (\ltwoprojection f^m - f_h^{m+1;\eta}),
      \discreteL^a (\partial_\tau e_h^{m+1} + \partial_\tau e_h^m)
    \rangle 
    \\
    \leq 
    \bigl(
      \| \discreteL^a (\ltwoprojection f^m - \interpolation f^m)\|
      +
      \| \discreteL^a (\interpolation f^m - f_h^{m+1;\eta})\|
    \bigr)
    \|\discreteL^a (\partial_\tau e_h^{m+1} + \partial_\tau e_h^m)\|.\label{eq:cons_weak_rhs_split}
  \end{multline}
  For the first term in the sum we observe that
  $\interpolation f^m \in V_h$, so that
  $\ltwoprojection f^m - \interpolation f^m
  = \ltwoprojection ( f^m - \interpolation f^m )$, and apply
  \eqref{eq:fully_discrete_fractional_stability} at $a\in\{0,\tfrac14\}$ followed by the approximation
  estimate \eqref{eq:fully_discrete_approximation}, obtaining
  \begin{equation}\label{eq:cons_weak_rhs_space}
    \| \discreteL^a (\ltwoprojection f^m - \interpolation f^m)\|
    {}\lesssim
    \| \L^a ( f^m - \interpolation f^m)\|
    \lesssim h^{p+1-2a},
  \end{equation}
  \noindent where the last step uses
  \eqref{eq:fully_discrete_approximation} with $s = p+1$, which is admissible by the
  assumed regularity $f \in L^\infty(0,T;H^{p+1}(\domain) \cap H_0^1(\domain))$.
  For the second term we use that $f_h^{m} = \interpolation f(m\tau)$ for every $m$, so that
  $\interpolation f^m - f_h^{m+1;\eta} = \interpolation ( f^m - f^{m+1;\eta} )$, and
  apply \eqref{eq:fully_discrete_fractional_stability} at $a\in\{0,\tfrac14\}$ directly to this discrete
  function, followed by a Taylor approximation of the source term:
  \begin{equation}\label{eq:cons_weak_rhs_time}
    \| \discreteL^a (\interpolation f^m - f_h^{m+1;\eta})\|
    {}=
    \| \discreteL^a \interpolation ( f^m - f^{m+1;\eta} ) \|
    \lesssim
    \| \L^a ( f^m - f^{m+1;\eta} ) \|
    \lesssim
    {\tau^2} \eta \| \L^a \partial_t^2 f^m\|
    \lesssim \tau^2,
  \end{equation}
  where the last step uses the assumed regularity
  $\partial_t^2 f \in L^\infty ( 0,T; D ( \L^{\frac14} ) )$.
  Similarly, we may treat the second defect~$\delta_{h;1}^m$. We use a Taylor expansion of the solution, stability and approximation estimates to obtain 
  \begin{multline}
    \tfrac{1}{2}
    \langle
      \delta_{h;1}^m, \discreteL^{2a} (\partial_\tau e_h^{m+1} + \partial_\tau e_h^m)
    \rangle
    =
    \tfrac{1}{2}
    \langle
      \discreteL^a (\partial_\tau^2\ritzprojection \hat u^{m+1} - \ltwoprojection \partial_t^2 \hat u^m),
      \discreteL^a (\partial_\tau e_h^{m+1} + \partial_\tau e_h^m)
    \rangle
    \\
    \leq \tfrac12
    \bigl(
      \| \discreteL^a (\ritzprojection - \ltwoprojection) \partial_\tau^2 \hat u^{m+1} \|
      +
      \| \discreteL^a \ltwoprojection (\partial_\tau^2 \hat u^{m+1} - \partial_t^2 \hat u^m)\|
    \bigr)
    \|\discreteL^a (\partial_\tau e_h^{m+1} + \partial_\tau e_h^m)\|,\label{eq:cons_weak_sol_split}
  \end{multline}
  where we added and subtracted
  $\ltwoprojection \partial_\tau^2 \hat u^{m+1}$ and used
  $\partial_\tau^2 \ritzprojection \hat u^{m+1} = \ritzprojection \partial_\tau^2 \hat u^{m+1}$.
  Note that $\ritzprojection$ is applied only to $\partial_\tau^2 \hat u^{m+1}$, a linear
  combination of $\hat u^{m+1}, \hat u^{m}, \hat u^{m-1} \in D ( \L^{\frac34} ) \subset
  H_0^1(\domain)$, where it is defined. For the first term, $(\ritzprojection -
  \ltwoprojection) \partial_\tau^2 \hat u^{m+1} \in V_h$, so
  \eqref{eq:fully_discrete_fractional_stability} at $a\in\{0,\tfrac14\}$ passes to the continuous norm
  and \eqref{eq:fully_discrete_approximation} with $s=p+1$ applies to each summand,
  \begin{equation}
    \| \discreteL^a (\ritzprojection - \ltwoprojection) \partial_\tau^2 \hat u^{m+1} \|
    \lesssim
    \| \L^a (\ritzprojection \partial_\tau^2 \hat u^{m+1} - \partial_\tau^2 \hat u^{m+1})\|
    +
    \| \L^a (\partial_\tau^2 \hat u^{m+1} - \ltwoprojection \partial_\tau^2 \hat u^{m+1})\|
    \lesssim h^{p+1-2a},\label{eq:cons_weak_sol_space}
  \end{equation}
  because $\partial_\tau^2 \hat u^{m+1}$ is a weighted average of
  $\partial_t^2 u$ over $[t_{m-1},t_{m+1}]$, whence
  $| \partial_\tau^2 \hat u^{m+1} |_{H^{p+1}(\domain)} \leq
  \| \partial_t^2 u \|_{L^\infty(0,T;H^{p+1}(\domain))}$,
  and, for the temporal part,
  \begin{equation}
    \| \discreteL^a \ltwoprojection ( \partial_\tau^2 \hat u^{m+1} - \partial_t^2 \hat u^{m} ) \|
    \lesssim
    \| \L^a ( \partial_\tau^2 \hat u^{m+1} - \partial_t^2 \hat u^{m} ) \|
    \lesssim
    {\tau^2} \| \L^a \partial_t^4 \hat u \|_{L^\infty(L^2)}
    \lesssim \tau^2,\label{eq:cons_weak_sol_time}
  \end{equation}
  \end{subequations}
  where the first step is \eqref{eq:fully_discrete_fractional_stability} at
  $a\in\{0,\tfrac14\}$ for $\ltwoprojection$ and the second is the standard consistency error of the
  central second difference, written in integral form so that only
  $\partial_t^4 u \in L^\infty ( 0,T; D ( \L^{\frac14} ) )$ is used. \eqref{eq:cons_weak_estimate} yields exactly the bound~\eqref{eq:cons_weak} for $\delta_{h;\frac14}^m$. 
  
  The next step consists of estimating the remaining defects 
  \(
    \delta_{h;\frac34}^m
    =
    \delta_{h;2}^m
    +
    \delta_{h;3}^m
    +
    \delta_{h;4}^m
  \)
  for $m\in\{1,n\}$. 
  Consider defect~$\delta_{h;2}^m$ for $m\in\{1,n\}$, then, by Taylor expansion we have that 
  \begin{subequations}\label{eq:cons_stiff_end_estimate}
  \begin{align}\label{eq:cons_stiff_end_time}
    \|
      \discreteL^{a - \frac12} \delta_{h;2}^m
    \|
    &=
    \Gamma_h^{m;\gamma}(u_h^{m+1})
    \|
      \discreteL^{\frac12 + a} \ritzprojection \theta (\hat u^{m+1} - 2\hat u^m + \hat u^{m-1})
    \|
    \leq \tau^2 C_\Gamma \theta 
    \| \discreteL^{\frac12 + a} \ritzprojection \partial_t^2 \hat u^m \|.
  \end{align}
  For the defect~$\delta_{h;3}^m$ for $m\in\{1,n\}$, we use Taylor expansion on the norm of the solution 
  \begin{align}\label{eq:cons_stiff_end_norm}
    \|
      \discreteL^{a - \frac12} \delta_{h;3}^m
    \|
    &=
        \beta^m \bigl| \gamma \|\nabla \hat u^{m+1}\|^2
        - 2\gamma \|\nabla \hat u^{m}\|^2
        + \gamma \|\nabla \hat u^{m-1}\|^2 \bigr|
      \|
        \discreteL^{\frac12 + a} \ritzprojection \hat u^m
      \|\notag
      \\
      &\leq
      \tau^2\beta^m\gamma 
      \bigl(
        \partial_t^2 \|\nabla \hat u^{m}\|^2
      \bigr)
      \|
        \discreteL^{\frac12 + a} \ritzprojection \hat u^m
      \|.
  \end{align}
  Finally, the last consistency defect $\delta_{h;4}^m$ is the projection part of the coefficient. By Ritz
  orthogonality, $\langle \nabla ( \hat u^m - \ritzprojection \hat u^m ), \nabla \varphi_h \rangle = 0$
  for all $\varphi_h \in V_h$, so that the quantities $\rho^m$ satisfy 
  \begin{equation*}
    \rho^m
    =
    \| \nabla \ritzprojection \hat u^{m} \|^2 - \| \nabla \hat u^{m} \|^2
    =
    - \| \nabla ( \hat u^{m} - \ritzprojection \hat u^{m} ) \|^2 ,
    \qquad\text{hence}\qquad
    | \rho^m | \lesssim h^{2p},
  \end{equation*}
  by the $H^1$-approximation estimate for $\ritzprojection$ with $s=p+1$. Since
  the weights $\gamma, 1-2\gamma, \gamma$ are non-negative and sum to $1$, the same bound holds
  for their combination, and therefore for $m\in\{1,n\}$ we have that
  \begin{align}\label{eq:cons_stiff_end_gamma}
    \|
      \discreteL^{a - \frac12} \delta_{h;4}^m
    \|
    &=
        \beta^m \bigl| \gamma \rho^{m+1}
        + (1 - 2\gamma) \rho^{m}
        + \gamma \rho^{m-1} \bigr|
      \|
        \discreteL^{\frac12 + a} \ritzprojection \hat u^m
      \|\notag
      \\
      &\leq
      \beta^m \max_{\mu \in \{ m-1,m,m+1 \}} |\rho^\mu| 
      \|
        \discreteL^{\frac12 + a} \ritzprojection \hat u^m
      \|
      \lesssim
      \beta^m
      h^{2p}
      \|
        \discreteL^{\frac12 + a} \ritzprojection \hat u^m
      \|.
  \end{align}
  \end{subequations}
  All right-hand sides are controlled by the chain
  \begin{equation}\label{eq:cons_ritz_chain}
    \| \discreteL^{\frac12 + a} \ritzprojection \partial_t^{j} \hat u^m \|
    \leq
    \lambda_1^{a - \frac14} \, \| \discreteL^{\frac34} \ritzprojection \partial_t^{j} \hat u^m \|
    \leq
    \lambda_1^{a - \frac14} \, \| \L^{\frac34} \partial_t^{j} \hat u^m \|
    \lesssim 1 ,
    \qquad j = 0,\dots,3 ,
  \end{equation}
  valid for $0 \leq a \leq \tfrac14$, whose first step is~\eqref{eq:discrete_poincare} and
  whose second is~\eqref{eq:fully_discrete_fractional_stability_ritz} at the
  exponent~$\tfrac34$. The final bound is \Cref{ass:existence_exact_solution} for $j=0$, which
  gives $\| \L^{\frac34} \hat u^m \| \leq \tfrac12 M_{\frac34}$, and the assumed temporal
  regularity $\partial_t^{j} u \in L^\infty ( 0,T;D ( \L^{\frac34} ) )$ for $j=1,2,3$. Applied
  with $j = 0,2$, this turns~\eqref{eq:cons_stiff_end_estimate}
  into~\eqref{eq:cons_stiff_end}.
  
  With the end points from the summation by parts estimated, we turn to the time derivatives of the last defects 
  \(
    \partial_\tau \delta_{h;\frac34}^m
    =
    \partial_\tau \delta_{h;2}^m
    +
    \partial_\tau \delta_{h;3}^m
    +
    \partial_\tau \delta_{h;4}^m
  \)
  for $m = 2,\dots,n$. The ideas are completely analogous to \eqref{eq:cons_stiff_end_estimate} with the addition that due to the assumed temporal regularity we obtain an additional order~$\tau$ which cancels the~$\tau^{-1}$ in the definition of~$\partial_\tau$. %
  For $\partial_\tau \delta_{2;h}^m$ for $2\leq m\leq n$, by Taylor expansion, the estimate \eqref{eq:cons_gamma_diff}, and $\Gamma_h^{n;\gamma}(u_h^{n+1}) \leq C_\Gamma$ by the bound established in the
  proof of \Cref{lem:fully_discrete_local_existence}, we have that 
  \begin{subequations}\label{eq:cons_stiff_estimate}
  \begin{alignat}{1}
    \|
      \discreteL^{a - \frac12} \partial_\tau\delta_{h;2}^m
    \|
    &\leq \tau^{-1} \theta 
    \bigl(
        | \Gamma_h^{m;\gamma}(u_h^{m+1}) - \Gamma_h^{m-1;\gamma}(u_h^{m}) |
        \| \discreteL^{\frac12 + a} \ritzprojection (\hat u^{m+1} - 2\hat u^m + \hat u^{m-1})
        \|\notag
        \\ 
        & \quad  +
        \Gamma_h^{m-1;\gamma}(u_h^{m})
        \| \discreteL^{\frac12 + a} \ritzprojection (\hat u^{m+1} - 3\hat u^m + 3\hat u^{m-1} - \hat u^{m-2})
        \|
    \bigr)\notag
    \\
    &\lesssim \tau^2 \theta 
    \bigl(\| \discreteL^{\frac12 + a} \ritzprojection \partial_t^2 \hat u^m \|
    + C_\Gamma \| \discreteL^{\frac12 + a} \ritzprojection \partial_t^3 \hat u^m \|
    \bigr).
  \end{alignat}
  For the derivative of $\partial_\tau \delta_{h;3}^m$ for $2\leq m\leq n$ using Taylor expansion of~$\beta$ for the first term, Taylor expansion of the norm of the solution for all three terms, and Taylor expansion of the solution for the last term, we have that 
  \begin{multline}
    \|
      \discreteL^{a - \frac12} \partial_\tau \delta_{h;3}^m
    \|
    \leq \tau^{-1}
    \Bigl[
        |\beta^m - \beta^{m-1}|\, |\gamma \|\nabla \hat u^{m+1}\|^2
        - 2\gamma \|\nabla \hat u^{m}\|^2
        + \gamma \|\nabla \hat u^{m-1}\|^2|\,
      \|
        \discreteL^{\frac12 + a} \ritzprojection \hat u^m
      \|\notag
      \\
      \qquad +
        \beta^{m-1} |\gamma \|\nabla \hat u^{m+1}\|^2
        - 3\gamma \|\nabla \hat u^{m}\|^2
        + 3\gamma \|\nabla \hat u^{m-1}\|^2
        - \gamma \|\nabla \hat u^{m-2}\|^2|\,
      \|
        \discreteL^{\frac12 + a} \ritzprojection \hat u^m
      \|\notag
      \\
      \qquad +
        \beta^{m-1} |\gamma \|\nabla \hat u^{m}\|^2
        - 2\gamma \|\nabla \hat u^{m-1}\|^2
        + \gamma \|\nabla \hat u^{m-2}\|^2|\,
      \|
        \discreteL^{\frac12 + a} \ritzprojection (\hat u^m - \hat u^{m-1})
      \|\notag
      \Bigr]
      \\
      \lesssim \tau^2 
      \bigl[
      (\partial_t^2 \|\nabla \hat u^{m}\|^2) \| \discreteL^{\frac12 + a} \ritzprojection \hat u^m \|
      +
       (\partial_t^3 \|\nabla \hat u^{m}\|^2) \| \discreteL^{\frac12 + a} \ritzprojection \hat u^m \|
      +
       (\partial_t^2 \|\nabla \hat u^{m-1}\|^2) \| \discreteL^{\frac12 + a} \ritzprojection \partial_t \hat u^m \|
      \bigr].
  \end{multline}
  Finally, for the derivative of $\partial_\tau \delta_{h;4}^m$ for $2\leq m\leq n$, we have that 
  \begin{equation*}
    | \rho'(t) |
    =
    2 \bigl| \bigl\langle \nabla ( I - \ritzprojection ) u(t) , \nabla ( I - \ritzprojection ) \partial_t u(t) \bigr\rangle \bigr|
    \leq
    2 \| \nabla ( I - \ritzprojection ) u(t) \| \, \| \nabla ( I - \ritzprojection ) \partial_t u(t) \|
    \lesssim
    h^{2p} ,
  \end{equation*}
  which allows us to use Taylor expansion of $\rho$. This, Taylor expansion of~$\beta$ for the first term, and Taylor expansion of the solution for the last term we obtain 
  \begin{align}
    \|
      \discreteL^{a - \frac12} \partial_\tau \delta_{h;4}^m
    \|
    &\leq \tau^{-1}
    \bigl(
        |\beta^m - \beta^{m-1}|\, |\gamma \rho^{m+1}
        + (1 - 2\gamma) \rho^{m}
        + \gamma \rho^{m-1}|\,
      \|
        \discreteL^{\frac12 + a} \ritzprojection \hat u^m
      \|\notag
      \\
      &\qquad  +
        \beta^{m-1} |\gamma (\rho^{m+1} - \rho^{m})
        + (1 - 2\gamma) (\rho^{m} - \rho^{m-1})
        + \gamma (\rho^{m-1} - \rho^{m-2})|\,
      \|
        \discreteL^{\frac12 + a} \ritzprojection \hat u^m
      \|\notag
      \\
      &\qquad  +
        \beta^{m-1} |\gamma \rho^{m}
        + (1 - 2\gamma) \rho^{m-1}
        + \gamma \rho^{m-2}|\,
      \|
        \discreteL^{\frac12 + a} \ritzprojection (\hat u^m - \hat u^{m-1})
      \|
      \bigr)\notag
      \\
      &\leq
      \beta_\mathrm{max} 
      ( h^{2p} \| \discreteL^{\frac12 + a} \ritzprojection \hat u^m \|
      +
      h^{2p} \| \discreteL^{\frac12 + a} \ritzprojection \hat u^m \|
      +
      h^{2p} \| \discreteL^{\frac12 + a} \ritzprojection \partial_t \hat u^m \|
       )
  \end{align}
  \end{subequations}
    \eqref{eq:cons_stiff_time} now follows from~\eqref{eq:cons_stiff_estimate}
  and~\eqref{eq:cons_ritz_chain}, applied with $j = 0,\dots,3$.

It remains to treat the stabilization term $S_h^m$ for $1\leq m\leq n$. 
From the
definition
\(
e_h^m
=
\ritzprojection \widehat{u}^{m}
-
u_h^m
\)
and the self-adjointness of $\discreteL$, we obtain
\begin{align*}
|\sigma^m|
&=
|
\|\nabla u_h^m\|^2
-
\|\nabla \ritzprojection\widehat{u}^{m}\|^2
|
=
|
\langle
\discreteL^{\frac12}
(
u_h^m+\ritzprojection\widehat{u}^{m}
),
\discreteL^{\frac12}
(
u_h^m-\ritzprojection\widehat{u}^{m}
)
\rangle
|
\\
&\lesssim
(
\|\discreteL^{\frac12 - a}u_h^m\|
+
\|\discreteL^{\frac12 - a}\ritzprojection\widehat{u}^{m}\|
)
\|\discreteL^{\frac12 + a}e_h^m\|
\lesssim
\|\discreteL^{\frac12 + a}e_h^m\|,
\end{align*}
where we used~\eqref{eq:discrete_poincare}, the uniform bounds on the discrete solution, and \eqref{eq:fully_discrete_fractional_stability_ritz} at $a=0$, and the uniform bounds on the discrete solution together with
the bound \eqref{eq:fully_discrete_fractional_stability_ritz_below_half} at $a=\tfrac14$, where the norm of the solution is controlled by
\Cref{ass:existence_exact_solution}.
Moreover, by the commutation identity \eqref{eq:ritz_commutation}, we have
\(
\discreteL^{1 + a}\ritzprojection\widehat{u}^{m}
=
\discreteL^{a}\ltwoprojection \L\widehat{u}^{m}.
\)
By the continuous equation at the time step $t_m$, we obtain 
\[
\L\widehat{u}^{m}
=
(q^m)^{-1}
(f^m-\partial_t^2\widehat{u}^{m}).
\]
Since $\alpha^m\geq\alpha_{\min}>0$, the assumed regularity and the stability of
$\ltwoprojection$ imply
\begin{align*}
\|\discreteL^{1 + a}\ritzprojection\widehat{u}^{m}\|
&=
\|\discreteL^{a}\ltwoprojection \L\widehat{u}^{m}\|
\lesssim
\|\L^{a}\L\widehat{u}^{m}\|
\lesssim
\|\L^{a}f^m\|
+
\|\L^{a}\partial_t^2\widehat{u}^{m}\|
\lesssim 1.
\end{align*}
Thus, for the stabilization term $S_h^m$, we have
\begin{equation*}
|
\langle
\discreteL^a S_h^m,
\discreteL^a \varphi_h
\rangle
|
\lesssim
(
\|\discreteL^{\frac12 + a} e_h^{m+1}\|
+
\|\discreteL^{\frac12 + a} e_h^m\|
+
\|\discreteL^{\frac12 + a} e_h^{m-1}\|
)
\|\discreteL^a \varphi_h\|.
\end{equation*}
This is precisely the last estimate~\eqref{eq:cons_stab} we needed to establish. 
\end{proof}

\subsection{Projection and interpolation operators}\label{app:projections}

\noindent We prove here the estimates collected
in~\Cref{sec:full-discretization} and record the references for the standard facts they rest
on. The approximation estimates of~\eqref{eq:fully_discrete_approximation} for
$\ltwoprojection$ and $\ritzprojection$, and the $L^2$-estimate for $\ritzprojection$ obtained
by the Aubin--Nitsche duality argument, are classical, see \cite[Ch.~0, 4 \& 5]{BrennerScott2008}
or \cite{KnabnerAngermann2021}; the quasi-interpolation operator $\interpolation$ is the one
of~\cite{Clement1975} for $p=1$ and, for general $p$ and with the homogeneous boundary condition
built in, the one of~\cite{ErnGuermond2017}, both of which are $L^2$- and $H_0^1$-stable and
satisfy that estimate by construction. The stability we use is stated separately for each
operator rather than as a blanket assertion: $\ltwoprojection$ is $L^2$-stable with constant
$1$, and its $H_0^1$-stability holds on quasi-uniform families by~\cite{CrouzeixThomee1987} and,
more generally, on shape-regular families under a local mesh condition
by~\cite{BramblePasciakSteinbach2002}; $\ritzprojection$ is $H_0^1$-stable with constant $1$,
and in weaker norms we use only~\eqref{eq:fully_discrete_approximation},
\eqref{eq:fully_discrete_fractional_stability_ritz} and the commutation
identity~\eqref{eq:fully_discrete_discrete_laplacian_equality}; $\interpolation$ is $L^2$- and
$H_0^1$-stable by construction.

\noindent The fractional estimate of \eqref{eq:fully_discrete_approximation} follows from the
first by interpolation: for $w \in H^1_0(\domain)$ we have
$\| \L^{\frac14} w \|^2 = \langle \L^{\frac12} w, w \rangle \leq \| \nabla w \| \, \| w \|$, and
applying this to $w = v - \Pi v$ and inserting the first estimate gives
$\| \L^{\frac14} ( v - \Pi v ) \| \lesssim h^{s-\frac12} | v |_{H^s(\domain)}$.

\noindent For~\eqref{eq:discrete_poincare}, every eigenvalue of $\discreteL$ is a Rayleigh
quotient over $V_h \subset H_0^1(\domain)$, hence bounded below by $\lambda_1$ independently of
$h$, and the spectral calculus on $V_h$ gives the first claim; combining it with
\eqref{eq:fully_discrete_fractional_stability_ritz} at $a = \tfrac12$ gives \eqref{eq:fully_discrete_fractional_stability_ritz_below_half}. For
\eqref{eq:discrete_inverse}, the inverse inequality states that the largest eigenvalue of
$\discreteL$ satisfies $\lambda_{\max,h} \leq \cstinv^2 h^{-2}$, and the spectral calculus again
gives the claim.

\begin{proof}[Proof of \eqref{eq:fully_discrete_fractional_stability} and
\eqref{eq:fully_discrete_fractional_stability_ritz}]
  Throughout we use that
  $\| \discreteL^{\frac12} \varphi_h \| = \| \nabla \varphi_h \|$ for $\varphi_h \in V_h$, and
  that $D ( \L^{a} )$ and $D ( \discreteL^{a} )$ are the interpolation scales generated
  by the two self-adjoint positive definite operators $\L$ and $\discreteL$, so that a
  linear map bounded at two exponents is bounded at every exponent in between, with the
  interpolated constant.

  For \eqref{eq:fully_discrete_fractional_stability} it suffices, by interpolation, to
  treat $a=0$ and $a=\tfrac12$. Both $\ltwoprojection$ and $\interpolation$ are defined on
  $L^2(\domain)$, so both endpoints are meaningful. At $a=0$ the assertion is their
  $L^2$-stability, and at $a=\tfrac12$ it reads
  $\| \nabla \ltwoprojection v \| + \| \nabla \interpolation v \| \lesssim \| \nabla v \|$,
  which is their $H^1_0$-stability; both hold by the references recorded above, and the
  interpolation of the two fractional-power scales used here is the standard interpolation of
  Hilbert scales \cite{BerghLofstrom1976}.

  For \eqref{eq:fully_discrete_fractional_stability_ritz} we again interpolate, now between
  $a=\tfrac12$ and $a=1$. At $a=\tfrac12$,
  \begin{equation*}
    \| \discreteL^{\frac12} \ritzprojection v \|
    =
    \| \nabla \ritzprojection v \|
    \leq
    \| \nabla v \|
    =
    \| \L^{\frac12} v \|\,,
  \end{equation*}
  \noindent again because the Ritz projection is the $H^1_0$-orthogonal projection onto
  $V_h$. At $a=1$ we use the defining property of $\ritzprojection$: for every
  $\varphi_h \in V_h$ and $v \in D ( \L )$,
  $\langle \discreteL \ritzprojection v, \varphi_h \rangle
  = \langle \nabla \ritzprojection v, \nabla \varphi_h \rangle
  = \langle \nabla v, \nabla \varphi_h \rangle
  = \langle \L v, \varphi_h \rangle
  = \langle \ltwoprojection \L v, \varphi_h \rangle$, that is,
  $\discreteL \ritzprojection v = \ltwoprojection \L v$,
  which is~\eqref{eq:fully_discrete_discrete_laplacian_equality}, whence $\| \discreteL \ritzprojection v \| = \| \ltwoprojection \L v \|
  \leq \| \L v \|$. Both endpoints hold with constant $1$, so does the interpolant.
\end{proof}

\noindent For $\varphi_h \in V_h$ we have $\ltwoprojection \varphi_h = \varphi_h$, so
\eqref{eq:fully_discrete_fractional_stability} reads
$\| \discreteL^{a} \varphi_h \| \lesssim \| \L^{a} \varphi_h \|$ for
$0 \leq a \leq \tfrac12$, which is the form in which it is used below.

\subsection{The three-layer displacement energy}\label{app:three_layer}

\begin{proof}[Proof of \Cref{lem:three_layer_displacement_energy}]
  Non-negativity of $\mathcal{W}_h^m$ holds by definition, both
  weights being non-negative for $\theta \geq \tfrac14$. %
  Expanding the squares in the definition of $\mathcal{W}_h^m$ gives the equivalent form
  \begin{equation}\label{eq:three_layer_expanded}
    \mathcal{W}_h^m
    =
    \theta \bigl( \| v_h^m \|^2 + \| v_h^{m-1} \|^2 \bigr)
    +
    ( 1 - 2\theta ) \langle v_h^m, v_h^{m-1} \rangle .
  \end{equation}
  Since $\theta \leq \tfrac12$ we have
  $1 - 2\theta \geq 0$, so applying
  $| \langle v_h^m, v_h^{m-1} \rangle | \leq \tfrac12 ( \| v_h^m \|^2 + \| v_h^{m-1} \|^2 )$
  to~\eqref{eq:three_layer_expanded} in both directions yields
  \begin{equation*}
    \tfrac{4\theta-1}{2} ( \| v_h^m \|^2 + \| v_h^{m-1} \|^2 )
    \leq
    \mathcal{W}_h^m
    \leq
    \tfrac12 ( \| v_h^m \|^2 + \| v_h^{m-1} \|^2 ) ,
  \end{equation*}
  which is~\eqref{eq:three_layer_upper}, and, after discarding the term
  $\| v_h^{m-1} \|^2$ on the left, also~\eqref{eq:three_layer_lower} whenever $4\theta - 1 > 0$. %
  For~\eqref{eq:three_layer_endpoint} split $v_h^m$ into its half-sum and half-difference,
  $v_h^m = \tfrac12 ( v_h^m + v_h^{m-1} ) + \tfrac12 ( v_h^m - v_h^{m-1} )$, and using
  $\| x + y \|^2 \leq 2 \| x \|^2 + 2 \| y \|^2$,
  \begin{equation*}
    \| v_h^m \|^2
    \leq
    \tfrac12 \| v_h^m + v_h^{m-1} \|^2 + \tfrac12 \| v_h^m - v_h^{m-1} \|^2
    \leq
    2 \mathcal{W}_h^m + \tfrac12 \| v_h^m - v_h^{m-1} \|^2\,.
  \end{equation*}
  For the remaining difference, we have that 
  $v_h^m - v_h^{m-1} = \tau \discreteL^{\frac12} ( \discreteL^{a} \partial_\tau u_h^{m} )$ with
  $\discreteL^{a} \partial_\tau u_h^{m} \in V_h$, so~\eqref{eq:discrete_inverse} with the pair
  $(0,\tfrac12)$ gives
  $\| v_h^m - v_h^{m-1} \| \leq \cstinv {\tau}{h^{-1}} \| \discreteL^{a} \partial_\tau u_h^{m} \|$,
  and~\eqref{eq:three_layer_endpoint} follows. %
  Finally, we turn to \eqref{eq:three_layer_alt}. By algebraic manipulation, we have for the left-hand side 
  \begin{align*}
    \langle 
        v_h^{m+1;\theta},
        v_h^{m+1} - v_h^{m-1}
    \rangle
    &=
    \tfrac14
    \langle 
        v_h^{m+1} + 2v_h^{m} + v_h^{m-1},
        v_h^{m+1} + v_h^m - (v_h^m + v_h^{m-1})
    \rangle
    \\
    &\quad + (\theta - \tfrac14)
    \langle 
        v_h^{m+1} - 2v_h^{m} + v_h^{m-1},
        v_h^{m+1} - v_h^m + (v_h^m - v_h^{m-1})
    \rangle
    \\
    &=
    \mathcal{W}_h^{m+1} - \mathcal{W}_h^{m}.\qedhere
  \end{align*}
\end{proof}

\begin{remark}[Sharpness at the endpoint]\label[remark]{rem:three_layer_sharpness}
Estimate~\eqref{eq:three_layer_lower} degenerates as $\theta \downarrow \tfrac14$
and fails at the endpoint, where $\mathcal{W}_h^m = \tfrac14 \| v_h^m + v_h^{m-1} \|^2$ vanishes for
$v_h^m = - v_h^{m-1}$. The substitute is~\eqref{eq:three_layer_endpoint}, which controls
$\| v_h^m \|$ only together with the kinetic term and only up to the factor $\tau h^{-1}$. That factor
is sharp. At $\theta = \tfrac14$ the sawtooth $u_h^m = (-1)^m \varepsilon \psi_h$, with $\psi_h$
an eigenfunction of $\discreteL$ for the largest eigenvalue $\lambda_{\max,h}$, gives $\mathcal{W}_h^m = 0$
and quotient $\| v_h^m \| \big/ \| \discreteL^{a} \partial_\tau u_h^{m} \|
= \tfrac{\tau}{2} \lambda_{\max,h}^{\frac12} \sim \tau h^{-1}$. Every argument recovering
$\| v_h^m \|$ from the three-layer energy there must therefore accept that factor or impose a mesh
condition.
\end{remark}

\end{document}